\pdfoutput=1
\documentclass[12pt]{article}
\usepackage{amsmath}
\usepackage{amsfonts}
\usepackage{amssymb}
\usepackage{latexsym}
\usepackage[dvips]{graphicx}
\usepackage{float}
\usepackage{geometry}
\AtEndDocument{\bigskip{\footnotesize
\textsc{School of Mathematics (Zhuhai), Sun Yat-Sen University (Zhuhai), China}\par
\textit{E-mail address: }\texttt{maiweixiong@gmail.com}\par
\medskip
\textsc{Macau Institute of Systems Engineering, Macau University of Science and Technology, Macao, China}\par
\textit{E-mail address: }\texttt{tqian@must.edu.mo}

}}

\usepackage{lipsum}

\newtheorem{thm}{\bf Theorem}[section]
\newtheorem{theorem}[thm]{\bf Theorem}
\newtheorem{lem}[thm]{\bf Lemma}

\newtheorem{cor}[thm]{\bf Corollary}

\newenvironment{proof}{\noindent{\em Proof:}}{\quad \hfill$\Box$\vspace{2ex}}

\newenvironment{remark}{\noindent{\bf Remark}}{\vspace{2ex}}

\renewcommand{\theequation}{\arabic{section}.\arabic{equation}}

\begin{document}
 \title{The Fourier Type Expansions on Tubes\let\thefootnote\relax\footnotetext{This work was supported by Macau Government FDCT 079/2016/A2 and Multi-Year Research Grant of the University of Macau No. MYRG2016-00053-FST.}}
\author{
Weixiong Mai
and Tao Qian\thanks{Corresponding author}
}
\date{}
 \maketitle
 \begin{abstract}
 In view of recent developments of the study of reproducing kernel Hilbert spaces, in particular with the context the Hardy spaces on tubes, aspects of rational approximation for functions of finite energy in several complex and several real variables are developed.  
 \end{abstract}


 \tableofcontents

  \section{Introduction}
 The study of approximation of one complex and real variable has a long history (see for instance \cite{Wa,Wa1,Zygmund}). In particular, J. L. Walsh discussed problems of approximating holomorphic functions by rational functions in one complex variable. In the one complex or real variable contexts the studies are mainly based on orthogonal function systems. In the unit disc and in the upper-half complex plane cases rational orthogonal systems, or Takenaka-Malmquist (TM) systems, are unavoidable. The trigonometric system, for instance, is a particular case of a TM system: when all the system parameters take the zero value then the TM system reduces to the trigonometric system. Studies in the one variable cases are closely related to the topics interpolation and sampling, as well as uniqueness sets, etc. There are comparably much less results in multivariate approximations. Studies on several variables, as in the one variable cases, involve the complex analytic (or Cauchy type) structures of the underlying spaces. There are basically two different complex structures of which one is several complex variables and the other is Clifford algebra, the latter being mainly for several real variables. The two settings have different natures. As an example, multiple trigonometric series deals with approximation by multi-polynomials to functions defined on the $n$-torus that correspond to holomorphic functions of several complex variables inside the polydisc ${\mathbb D}^n.$ In the several complex variables setting approximations on boundaries of open sets of ${\mathbb C}^n$ have been studied. In the Clifford setting approximations to functions defined on manifolds as boundaries of open sets in Euclidean spaces have been studied. A model of the Clifford context is the conjugate harmonic systems in the sense of Stein-Weiss (\cite{SW}). The multiple variables cases induce hard and open problems such as Bochner-Riesz summation and non-restrictive rectangle summability, etc. The main obstacle with a practical and comprehensive theory of multivariate approximation would be the lack of TM-systems as rational orthogonal systems in one-complex variable. The present paper aims to develop a rational approximation theory and methodology in the several complex variables setting with the underlying spaces as the Hardy spaces on tubes. The latter also gives rise to approximation on ${\mathbb R}^n$. The study will be in the spirit of the AFD (adaptive Fourier decomposition) as given in \cite{QW}. With a simple and elegant manner, 1D-AFD, or Core-AFD, offers an effective rational approximation method in one complex variable. Due to its particular algebraic formulation the Gram-Schmidt (G-S) orthogonalization is combined in the process. By using the maximal selection principle rapid convergence is achieved. In higher dimensions there does not exist an algebraic formulation like Core-AFD of one-dimension, and thus the G-S process cannot be incorporated algebraically and has be down separately. 
 The effect of the algebraic approach of Core-AFD can, in fact, be achieved through a pre-orthogonal procedure. The latter, being called Pre-orthogonal AFD, or POAFD in brief (see \cite{Q-2D} and further in \cite{Qbook}), is considered as new formulation applicable to a large class of Hilbert spaces. It is, particularly, applicable, with equal force as AFD of the Hardy space of the unit disc, to a large class of reproducing kernel Hilbert spaces  (see \cite{MQ1}), including our several complex variables context.
  Below we will give an expository overview of Core-AFD, and show the passage leading AFD to POAFD. Before going into details we recall that by adopting the interpolation formulation of Blaschke products in higher dimensions AFD theory has been successfully generalized to approximate matrix-valued functions of multivariate complex variables (\cite{ACQS1,ACQS2}). The task of the present paper is to establish POAFD method for Hardy spaces on tubes through verifying the corresponding boundary vanishing conditions. This establishment raises the concept and introduces methodology of rational approximation in Hardy spaces on tubes generalizing the long existing classical theory in the disc and upper half complex plane. The Blaschke product approach, however, seems to be more delicate and restrictive. The POAFD method may be adopted to a wider class of spaces (\cite{QD}).

 Recall that ${\mathbb D}$ denotes the unit disc in the complex plane. Let $H^2({\mathbb D})$ denote the classical Hardy $H^2$-space in the unit disc. Among several equivalent definitions of $H^2({\mathbb D})$ we adopt the one in terms of the coefficients of the Taylor expansion of a complex holomorphic function in the disc, that is,
 \[ H^2({\mathbb D})=\{f:{\mathbb D}\to {\mathbb C}\ |\ f \ {\rm is\ holomorphic\ in\ }\ {\mathbb D}, f(z)=\sum_{k=0}^\infty c_kz^k, \sum_{k=0}^\infty |c_k|^2<\infty\}.\] We will be using the following dense subset of $H^2({\mathbb D})$:
 \[ \{e_a(z);a\in {\mathbb D}\}, \quad {\rm where}\quad e_a(z)=\frac{\sqrt{1-|a|^2}}{1-\overline{a}z}.\]
 We note that the reproducing kernel $k_a$ of $H^2({\mathbb D})$
 is
 \[ k_a(z)=\frac{1}{1-\overline{a}z},\]
 and $e_a$ is the normalized reproducing kernel.
 For any function $f_1=f\in H^2({\mathbb D})$ and $a_1\in {\mathbb D}$ one has the algebraic identity
 \[ f(z)=\langle f_1, e_{a_1}\rangle e_{a_1}(z)+\frac{f_1(z)-\langle f_1, e_{a_1}\rangle e_{a_1}(z)}{\frac{z-a_1}{1-\overline{a}_1z}}
 \frac{z-a_1}{1-\overline{a}_1z}.\]
 Setting
 \[f_2(z)=\frac{f_1(z)-\langle f_1, e_{a_1}\rangle e_{a_1}(z)}{\frac{z-a_1}{1-\overline{a}_1z}},\] the above identity can be re-written
 \begin{eqnarray}\label{maximum sifting}  f(z)=\langle f_1, e_{a_1}\rangle e_{a_1}(z)+f_2(z)\frac{z-a_1}{1-\overline{a}_1z}\end{eqnarray}
 with the energy relation
 \[ \|f\|^2=\| \langle f_1, e_{a_1}\rangle e_{a_1}\|^2+\|f_2\|^2 =(1-|a_1|^2)|f_1(a_1)|^2 + \|f_2\|^2,\]
 where besides the usual orthogonal projection property we also used the complex unit modular property of the M\"obius transform. The process of getting $f_2$ from $f_1$ with the parameter $a_1$ is called a sifting process. 
 The nice thing is that in the open set ${\mathbb D}$ one can find
 \[ a_1=\arg \max \{ (1-|a|^2)|f_1(a)|^2\ : \ a\in {\mathbb D}\}\]
 (\cite{QW}). The availability of selection of $a_1$ in (\ref{maximum sifting}) with the above maximal property is called the \emph{maximal selection principle}. By fixing such $a_1$ the energy $\|f_2\|^2$ is minimized. Repeating the same procedure to $f_2,$ and so on. Up to  the $n$-th step one gets
 \begin{eqnarray}\label{1} f(z)=\sum_{k=1}^m \langle f_k, e_{a_k}\rangle B_k(z) + f_{m+1}\prod_{k=1}^m\frac{z-a_k}{1-\overline{a}_kz}\end{eqnarray}
 where for $k=1,\cdots ,m,$
 \begin{eqnarray}\label{2} a_k=\arg \max \{ (1-|a|^2)|f_k(a)|^2\ : \ a\in {\mathbb D}\},\end{eqnarray}
 \begin{eqnarray}\label{3}f_k(z)=\frac{f_{k-1}(z)-\langle f_{k-1}, e_{a_{k-1}}\rangle e_{a_{k-1}}(z)}{\frac{z-a_{k-1}}{1-\overline{a}_{k-1}z}},\end{eqnarray}
 and\begin{eqnarray}\label{4}B_k(z)=B_{\{a_1,\cdots,a_k\}}(z)=
 \frac{\sqrt{1-|a_k|^2}}{1-\overline{a}_kz}\prod_{l=1}^{k-1}
 \frac{z-a_l}{1-\overline{a}_lz},\end{eqnarray}
 the latter being the so called
 rational orthonormal, or Takenaka-Malmquist, system which is generated by the algebraic sifting process and automatically orthonormal. Due to the maximal selections of the parameters it has fast convergence.
 The following theorem gives what we call adaptive Fourier decomposition (AFD or Core-AFD).
 \begin{theorem}[\cite{QW}]\label{AFD}
 	For any  given function $f$ in the Hardy $H^2$ space, by making a maximal selection at each step we have
 	\[  f(z)=\sum_{k=1}^\infty \langle f_k, e_{a_k}\rangle B_k(z).\]\end{theorem}
 We remark that due to the maximal selection principle in the sifting process AFD converges at fast pace. It is an equivalent type of fast convergence that we will develop in this paper for multivariate functions. A technically alternative adaptive expansion, also of the Fourier type, called unwinding Blaschke expansion, has been developed in a series of recent papers by Coifman et al. (\cite{CS,CSW,CP}), and, independently, by Qian et al. (\cite{Q1}). The unwinding method does not seem to have a close counterpart in higher dimensions, for it is strongly dependent on factorization. The second remark on 1D-AFD is that each term of a TM system, when prefixing $a_1=0$, has a positive boundary phase derivative function, defined as the instantaneous frequency function of the term. This aspect, in fact, was the motivation of the mono-component function theory, as well as AFD (\cite{QW}). Unwinding Blaschke expansions also have positive frequencies. In this paper we do not pursue the unwinding method but concentrate in POAFD as a generalization of AFD characterized by the \emph{pre-orthogonal maximal selection principle } (\ref{MSP-AFD}).

 AFD was originally established in the Hardy spaces of the unit disc and the upper half plane. For multivariate cases the step (\ref{3}), involving the generalized backward shift operation to get the \emph{induced remainder} $f_{k+1},$ is unavailable, and thus the maximal selection in step (\ref{2}) cannot be performed. Suggested by the relations
 \begin{eqnarray}\label{relations}
 \langle f_k,e_{a_k}\rangle=\langle g_k,B_k\rangle=\langle f,B_k\rangle,\end{eqnarray} where \[ g_{k}(z)=f_{k}(z)\prod_{l=1}^{k-1}\frac{z-a_l}{1-\overline{a}_lz}\] is the $k$-th \emph{standard remainder}, one can perform what is now called pre-orthogonal AFD, abbreviated as POAFD, the latter is reduced to AFD in the classical setting and thus is seen to have equal force as AFD in general reproducing kernel Hilbert spaces satisfying some boundary vanishing condition (BVC).
 To illustrate the method we do the following preparations.

 Let $\{a_1,\cdots,a_m,\cdots\}$ be a finite or infinite sequence. For a fixed $m$ we define the multiple of $a_m$  in $\{a_1,\cdots,a_m\},$  denoted by $l(a_m),$ as the repeating time of the value of $a_m$ in $\{a_1,\cdots,a_m\}$ (i.e. the cardinality of the set $\{j:a_j=a_m,j\leq m\}$). We accordingly define the \emph{multiple reproducing}\emph{ kernels } \begin{eqnarray}\label{tilde}\tilde{k}_{a_m}\triangleq
 \left[\left(\frac{\partial}{\partial\overline{a}}\right)^{l(a_m)-1}
 k_{a}\right]_{a=a_m}\triangleq\left(\frac{\partial}{\partial\overline{a}}\right)^{l(a_m)-1}
 k_{a_m}\end{eqnarray}
 with the property
 \begin{eqnarray}\label{partial} f^{(l)}(a)=\langle f,\left(\frac{\partial}{\partial \overline{a}}\right)^{l}k_a\rangle, \quad l=1,2,\cdots
 \end{eqnarray}

 Let $f\in H^2({\mathbb D})$ and $\{a_k\}_{k=1}^m$ be $m$ existing points in ${\mathbb D}.$ To formulate POAFD we replace (\ref{3}) and (\ref{2}) with the selection of $a_{m+1}$ in ${\mathbb D}$ to satisfy
 \begin{align}\label{MSP-AFD}
 a_{m+1}:= \arg \max_{a\in {\mathbb D}} |\langle f, B_{m+1}^a \rangle|,
 \end{align}
 where $B_{m+1}^a$ is defined by the requirement that $\{B_1,\cdots,B_m,B_{m+1}^a\}$ is the result of the G-S orthogonalization process applied to $\{B_1,\cdots,B_m,\tilde{k}_{a_m}\}.$ In such way one does not need to formulate the reduced remainders that are crucial in 1D-AFD but unavailable in the higher dimensions. The replacements of the reduced remainders resulted from the sifting process are the newly designed pre-orthogonal algorithm and use of multiple reproducing kernels. In each concrete reproducing kernel Hilbert space context one has to show that the boundary vanishing condition (BVC) holds that implies existence of an $a_{m+1}$ in (\ref{MSP-AFD}) (the pre-orthogonal maximal selection principle). The contribution of the present paper is realization of POAFD in various Hardy spaces on tubes, consisting of settings of the question in the formulation of reproducing kernel Hilbert spaces, specifying and verification of the boundary vanishing conditions, applications of the approximation to functions of several real variables, and generalizations of the theory to regular cones. 

 We next recall the necessary notations and terminologies for the Hardy spaces on tubes (\cite{SW}). Let $B$ be an open subset in $\mathbb R^n$. We say that $T_B$ is a
 tube over $B$, if each $z\in T_B\subset \mathbb C^n$ is of the form
 $
 z=x+iy, x\in\mathbb R^n, y\in B\subset \mathbb R^n.
 $   We will be, in particular, working on the case where $B$ is an open cone.
Open cones are nonempty open subsets $\Gamma\in \mathbb R^n$ satisfying
 (1) $0\not \in \Gamma$, and (2) whenever $x,y \in \Gamma$ and $a, b>0$ then $a x +b y \in \Gamma.$
 A closed cone is the closure of an open cone. It is clear that if $\Gamma$ is an open cone then $\Gamma^{*}=\{x\in \mathbb R^n: x\cdot t \geq 0, t\in \Gamma\}$ is a close set. If, in addition,  $\Gamma^{*}$ has a non-void interior, then $\Gamma$ is said to be a regular cone, and $\Gamma^{*}$ is called the cone dual to $\Gamma.$
For instance, for $n=1,$  the upper half-plane in the complex plane can be regarded as the tube $T_{(0,\infty)}=\{z=x+iy,x\in \mathbb R, y\in (0,\infty)\}.$
 Denote by $H^2(T_{\Gamma})$ the Hardy space on $T_{\Gamma}$. We
 say $F\in H^2(T_{\Gamma})$, if $F$ is holomorphic on
 $T_{\Gamma}$ and satisfies
 $$
 \Vert F \Vert^2=\sup_{y\in \Gamma}\int_{\mathbb{R}^n}\vert
 F(x+iy)\vert^2dx<\infty.
 $$
 $H^2(T_{\Gamma})$ is a Hilbert space equipped with the inner product
 $$
 \langle F, G \rangle =\int_{\mathbb R^n}F(\xi)\overline {G(\xi)} d\xi,
 $$
 where $F(\xi)=\lim_{\eta\in \Gamma,\eta\to 0}F(\xi+i\eta)$ is the limit function in the $L^2$-norm, as well as in the a.e. pointwise sense, and so is $G(\xi)$. Existence of such non-tangential boundary limits is a fundamental result of the Hardy spaces on tubes (see \cite{SW}).
 The space $H^2(T_{\Gamma})$ is a reproducing kernel Hilbert space with the reproducing kernel
 $$
 K(w,\overline z)= \int_{\Gamma^*}e^{2\pi i \omega \cdot
 	t}\overline{e^{2\pi iz \cdot t}}dt, \quad w,z\in T_{\Gamma}.
 $$
 In this paper we are mainly concerned with the following special tube
 $T_{\Gamma_1}$, where
 $$
 \Gamma_1=\{y\in \mathbb R^n: y_1>0, y_2>0,...,y_n>0\}.
 $$
 The theory and algorithm given in \cite{Q-2D} are restricted to the bi-disc $H^2(\mathbb D^2).$ This paper treats in several unbounded domains. As given in Section 3, our main result is\\

 \noindent {\bf Main Result} {\it
 	For $F\in H^2(T_{\Gamma_1})$, there holds
 	\begin{align}\label{convergence}
 	\lim_{m\to\infty}||F-\sum_{k=1}^m\langle F,\mathcal B_k\rangle\mathcal B_k||=0,
 	\end{align}\\
 	where, as formulated in (\ref{MSP-AFD}), each element of the sequence $\{z^{(k)}\}_{k=1}^{\infty}, z^{(k)}\in \mathbb{R}^n,$ is selected according to the maximal selection principle specified in (\ref{min_pro_revise}), $\{\mathcal B_k\}$ is obtained by applying the Gram-Schmidt orthogonalization process to the corresponding multiple Cauchy-Szeg\"o kernels.
 	
 }

\bigskip

 \noindent  The well-definedness of $\{\mathcal B_k\}$ will be proved in Section $3$. As a matter of fact, discussions of the system $\{\mathcal B_k\}$ constitute a main part of the paper.


 The optimal selections of $z^{(k)}$ towards the pre-orthogonal maximal principle (\ref{min_pro_revise}) are guaranteed by BVC. In this paper we will prove the validity of BVC in $H^2(T_{\Gamma_1})$, and show that the convergent rate of POAFD is $O(m^{-\frac{1}{2}})$ for $H^2(T_{\Gamma_1},M),$ a particular subclass of functions in $H^2(T_{\Gamma_1})$. As an application of the main result, rational approximation of functions in $L^2(\mathbb R^n)$ can be obtained through dividing the $L^2$ space to $2^n$ Hardy spaces on tubes, and then use the theory for the Hardy spaces (also see \cite{Q-2D}). 
 
 As a by-product, in Appendix A we will give another kind of rational approximation in $H^2(T_{\Gamma_1})$ by using BVC in $H^2(T_{\Gamma_1}).$ Also, we will study POAFD in Hardy spaces on tubes, $H^2(T_\Gamma)$, where $\Gamma$ are some regular cones. Since treatments in the $H^2(T_{\Gamma})$ case are more complicated than those in the $H^2(T_{\Gamma_1})$ case, we put all details of the $H^2(T_{\Gamma})$ case in Appendix B.
 In fact, through verifying BVC in $H^2(T_\Gamma)$ (see Theorem \ref{g-cone} in Appendix B), we have\\ \par

 {\it  If $\Gamma$ is a regular cone such that (\ref{CS-infty}) holds, then for $F\in H^2(T_\Gamma)$ there holds
 	\begin{align*}
 	\lim_{m\to\infty}||F-\sum_{k=1}^m\langle F,\mathcal B_k\rangle\mathcal B_k||=0,
 	\end{align*}
 	where each element of $\{z^{(k)}\}_{k=1}^{\infty}, z^{(k)}\in \mathbb{R}^n,$ is selected according to the pre-orthogonal maximal principle $(\ref{min_pro_revise1})$, and $\{\mathcal B_k\}$ is obtained by applying the Gram-Schmidt orthogonalization process to the corresponding Cauchy-Szeg\"o kernels.}\\
\noindent We note that the theory is dependent on the condition (\ref{CS-infty}).
 	For regular cones in $\mathbb R^2$ (that is, $n=2$), and for polygonal cones and circular cones in $\mathbb R^n$ for any $n\ge 2,$ we can prove the validity of (\ref{CS-infty}).

 The writing plan is as follows. In Section $2$ some related and basic results for $H^2(T_{\Gamma_1})$ are given. In Section $3$ we devote to establishing POAFD in $H^2(T_{\Gamma_1})$. In Section $4$ we prove the validity of BVC in $H^2(T_{\Gamma_1})$. In Section $5$ we investigate the convergent rate of POAFD and rational approximation of functions in $L^2(\mathbb R^n)$. In Appendix A we give another kind of rational approximation in $H^2(T_{\Gamma_1})$. In Appendix B we explore POAFD in $H^2(T_\Gamma)$ for $\Gamma$ being in some particular subclasses of regular cones.

 \section{Preliminaries}
 In this section, we will review fundamental properties of $H^2(T_{\Gamma})$. For more information, see e.g. \cite{SW,JFAK,Li-Deng-Qian}.

  As given previously, $$
 \Gamma_1=\{y\in \mathbb R^n: y_1>0, y_2>0,...,y_n>0\},
 $$
 whose dual cone is $\Gamma_1^*=\overline \Gamma_1$.

 For $H^2(T_\Gamma),$ we have
 \begin{thm}[Paley-Wiener Theorem {\cite[page 101]{SW}}]\label{PW}
 	Suppose $\Gamma$ is a regular cone. Then $F\in
 	H^2(T_{\Gamma})$ if and only if
 	$$
 	F(z)=\int_{\Gamma^*}e^{2\pi i z\cdot t}f(t)dt
 	$$
 	where $f$ is a measurable function on $\mathbb R^n$ satisfying
 	$$
 	\int_{\Gamma^*}|f(t)|^2dt <\infty.
 	$$
 	Furthermore,
 	$$
 	||F||=\left (\int_{\Gamma^*}|f(t)|^2dt \right
 	)^{\frac{1}{2}}.
 	$$
 \end{thm}
 $H^2(T_{\Gamma})$ is a reproducing kernel Hilbert space whose reproducing kernel is the
 Cauchy-Szeg\" o kernel
 $$
 K(w,\overline z)= \int_{\Gamma^*}e^{2\pi i \omega \cdot
 	t}\overline{e^{2\pi iz \cdot t}}dt, \quad w,z\in T_{\Gamma}.
 $$
 The corresponding Poisson-Szeg\"o kernel is given by
 $$
 P_{y}(x)=\frac{K(z,0)K(0, \overline z)}{K(z,\overline z)}.
 $$
 and $P_y(x)\in L^p,$ for $1\leq p\leq \infty$.\\
 We will use the notation $K_\Gamma(w,\overline z)$ if we want to emphasize the specific $\Gamma$ in the context.
 In particular, if $\Gamma=\Gamma_1$ then the Cauchy-Szeg\"o kernel and the Poisson-Szeg\"o kernel can be exactly computed by the following formulas
 $$
 K_{\Gamma_1}(w,\overline z)= \int_{\overline \Gamma_1}e^{2\pi i \omega \cdot
 	t}\overline{e^{2\pi iz \cdot t}}dt = \prod_{k=1}^n \frac{-1}{2\pi
 	i(\omega_k-\overline z_k)}
 $$
 and
 $$
 P_{y}(x)=\frac{K_{\Gamma_1}(z,0)K_{\Gamma_1}(0, \overline z)}{K_{\Gamma_1}(z,\overline z)}=
 \prod_{k=1}^n \frac{y_k}{\pi(x_k^2+y_k^2)}.
 $$
 For general $\Gamma$ the integral formulas corresponding to the Cauchy-Szeg\"o and
 the Poisson-Szeg\"o kernels are, respectively, given as in
 \begin{thm}[{\cite[page 103]{SW}}]\label{rf_s}
 	If $F\in H^2(T_{\Gamma})$ then
 	$$
 	F(z)=\int_{\mathbb R^n}F(\xi) \overline{K(\xi, \overline z)} d\xi =
 	\int_{\mathbb R^n}F(\xi) {K(z,  \xi)} d\xi
 	$$
 	for all $z=x+iy\in T_{\Gamma}$, where $F(\xi)= \lim_{\eta\to 0, \eta
 		\in \Gamma}F(\xi+i \eta)$ is the limit function in the $L^2$-norm.
 \end{thm}
 \begin{thm}[{\cite[page 106]{SW}}]\label{rf_p}
 	If $F\in H^2(T_{\Gamma})$, then
 	$$
 	F(z)=  \int_{\mathbb R^n}F(\xi) {P_y(x-\xi)} d\xi
 	$$
 	for all $z=x+iy\in T_{\Gamma}$, where $F(\xi)= \lim_{\eta\to 0, \eta
 		\in \Gamma}F(\xi+i \eta)$ is the limit function in the $L^2$-norm.
 \end{thm}

 \begin{thm}[{\cite[page 119]{SW}}]\label{g-rf-p}
 	Suppose $\Gamma$ is a regular cone in $\mathbb R^n$, and $F\in H^p(T_{\Gamma}), 1\leq p<\infty$, then
 	$$
 	\lim_{y\in\Gamma,y\to 0}\int_{\mathbb R^n}|F(x+iy)-F(x)|^pdx=0,
 	$$
 	and
 	$$
 	F(x+iy)=\int_{\mathbb R^n}F(\xi)P_y(x-\xi)d\xi,
 	$$
 	where $F(\xi)$ is the limit function, whose existence is in the norm sense, as well as in the subsequence-pointwise convergence sense (\cite[Chapter III,Theorem 5.5]{SW}).
 \end{thm}

 \section{POAFD in $H^2(T_{\Gamma_1})$}

 Suppose that $\{z^{(k)}\}_{k=1}^\infty$ is a sequence of distinct points in $T_{\Gamma_1}$. Under such assumption, $\{K_{\Gamma_1}(\cdot,\overline {z^{(k)}})\}_{k=1}^\infty$ are linearly
 independent.
 Let $\{ \mathcal B_k\}_{k=1}^m$ be the Gram-Schmidt (G-S)
 orthogonalization of $\{K_{\Gamma_1}(\cdot, \overline{z^{(k)}})\}_{k=1}^m$. For a function in $H^2(T_{\Gamma_1})$, One can define the $m$-th partial sum
 \begin{align}\label{F_orth}
 S_m(F)=\sum_{k=1}^m\langle F, \mathcal B_k \rangle
 \mathcal B_k.
 \end{align}
 Since $\text{span}\{\mathcal B_1,...,\mathcal B_m\}=\text{span}\{K_{\Gamma_1}(\cdot,\overline {z^{(1)}}),...,K_{\Gamma_1}(\cdot,\overline {z^{(m)}})\}$, and
 $F-S_m(F)$ is in the orthogonal complement of the span, we have
 \begin{align}\label{Sm-interpolation}
 F(z^{(k)})=\langle F, K_{\Gamma_1}(\cdot,\overline {z^{(k)}}) \rangle=\langle S_m(F), K_{\Gamma_1}(\cdot,\overline {z^{(k)}}) \rangle.
 \end{align}

 In the following discussion, we remove the restriction that all elements of $\{z^{(k)}\}_{k=1}^\infty$ are distinct from each other, i.e., there may exist $k\neq l$ such that $z^{(k)}=z^{(l)}.$ In such case, the original definition of $S_m(F)$ is meaningless in such situation. Therefore, we need to define the generalized $S_m(F)$ that is denoted by $\widetilde {S_m}(F)$. The generalization is intrinsically related to the POAFD methodology.

 Set $\phi_{z^{(k)}}=\phi_{k}=K_{\Gamma_1}(\cdot, \overline{z^{(k)}})$. By the G-S orthogonalization process, we have
 \begin{align}\label{G-S}
 \begin{split}
 \gamma_1&=\gamma_{\{z^{(1)}\}}= \phi_{1},\\
 \gamma_k&=\gamma_{\{z^{(1)},...,z^{(k)}\}}=\phi_{k}-\sum_{l=1}^{k-1}{\langle \phi_k, \frac{\gamma_l}{||\gamma_l||} \rangle}\frac{\gamma_l}{||\gamma_l||}, \quad k\geq 2\\
 \mathcal B_k &=\mathcal B_{\{z^{(1)},...,z^{(k)}\}}=\frac{\gamma_k}{||\gamma_k||}.
 \end{split}
 \end{align}
 To define $\widetilde {S_m}(F)$, it suffices to define $\{\mathcal B_k\}_{k=1}^m$ for the case that there may happen $z^{(k)}=z^{(l)}, k\neq l$.
 Such kind of discussion on $\{\mathcal B_k\}_{k=1}^m$ should be regarded as multiple poles Takenaka-Malmquist (TM) system, and is automatically motivated by 1D AFD or pre-orthogonal maximal selection principle in POAFD methods.
 In fact, it has been, respectively, made in the one complex variable \cite{QWe}, quaternionic analysis \cite{QSW} and several complex variables \cite{Q-2D} settings. The general discussion on such property of $\{\mathcal B_k\}_{k=1}^m$ is included in \cite{Q-2D}, and recently, a detailed proof is given in \cite{QD} and \cite{Qian-Advance}.

  The present context that we deal with is with the several complex variables setting. For simplicity, we interpret this for $z=(z_1, z_2)\in \mathbb C^2$.
 Unlike the one complex variable case, the generalization of TM systems in higher dimensions is more subtle since the dimension of the linear space spanned by the $h$-th order ($h\geq 1$) partial derivatives is more than $2.$ In this paper we provide a strategy that is in the spirit of the one dimensional case. Let $l_k$ be the cardinality of the set $\{j:z^{(j)}=z^{(k)}, j\leq k\}$. Suppose that $\vec d=(d_1,d_2)\in \mathbb P^2=\{\xi=(\xi_1,\xi_2)\in \mathbb C^2:|\xi_1|^2+|\xi_2|^2=1\},$ and $\vec d$ is fixed in the following discussion. As in the one dimensional case, we define
 \begin{align*}
 \widetilde \phi_{z^{(k)}}^{\vec d}=\frac{1}{(l_k-1)!}(\vec d\cdot \bigtriangledown)^{(l_k-1)}\phi_z|_{z=z^{(k)}},
 \end{align*}
 where $\vec d\cdot \bigtriangledown=d_1\frac{\partial}{\partial \overline z_1}+d_2\frac{\partial}{\partial \overline z_2}.$ Note that $\sum_{|\alpha|=l_k-1}\frac{\partial^\alpha\phi_{z}|_{z=z^{(k)}}}{\alpha!}\vec d^\alpha=\frac{1}{(l_k-1)!}(\vec d\cdot \bigtriangledown)^{l_k-1}\phi_z|_{z=z^{(k)}},$ where $\alpha=(\alpha_1,\alpha_2),$ $|\alpha|=\sum_{j=1}^2 \alpha_j,$ $\alpha!=\prod_{j=1}^2\alpha_j!$ and $\partial^\alpha\phi_z=\partial^{\alpha_1}\partial^{\alpha_2}\phi_z=\frac{\partial^{|\alpha|}\phi_z}{\partial \overline z_1^{\alpha_1}\partial \overline z_2^{\alpha_2}}.$
 Next we assume that $\{\mathcal B_k\}_{k=1}^m$ is the G-S orthogonalization of $\{\widetilde \phi_{z^{(k)}}^{\vec d}\}_{k=1}^m.$ Set $w=z^{(m)}+r\vec d$ with $\vec d\in \mathbb P^2.$ Since $\phi_z=K_{\Gamma_1}(\cdot,\overline z)$ is anti-analytic in $z,$ by the Taylor expansion with respect to the directional derivatives in the direction $\vec d,$ we have that
 \begin{align*}
 \phi_w &=\mathcal T^{(m)}(\phi_{z^{(m)}+r\vec d}) + \mathcal R^{(m)}(\phi_{z^{(m)}+r\vec d})\\
 &=\sum_{|\alpha|\leq l_m-1}\frac{\partial^\alpha \phi_z|_{z=z^{(m)}}}{\alpha !}(r\vec d)^\alpha + \sum_{|\alpha| \geq l_m}\frac{\partial^\alpha \phi_z|_{z=z^{(m)}}}{\alpha !}(r\vec d)^\alpha,
 \end{align*}
 and consequently,
 \begin{align*}
 \mathcal T^{(m)}(\phi_{z^{(m)}+r\vec d})=\sum_{j=1}^m \langle \mathcal T^{(m)}(\phi_{z^{(m)}+r\vec d}),\mathcal B_j\rangle \mathcal B_j.
 \end{align*}
 Then we have
 \begin{align*}
 &\lim_{r\to 0}\mathcal B_{\{z^{(1)},...,z^{(m)},w\}}\\
 &=\lim_{r\to 0}\frac{\phi_w-\sum_{j=1}^m\langle \phi_w,\mathcal B_j \rangle \mathcal B_j}{||\phi_w-\sum_{j=1}^m\langle \phi_w,\mathcal B_j \rangle \mathcal B_j||}\\
 &=\lim_{r\to 0}\frac{\phi_w-\mathcal T^{(m)}(\phi_{z^{(m)}+r\vec d})-\sum_{j=1}^m\langle \phi_w-\mathcal T^{(m)}(\phi_{z^{(m)}+r\vec d}),\mathcal B_j \rangle \mathcal B_j}{||\phi_w -\mathcal T^{(m)}(\phi_{z^{(m)}+r\vec d})-\sum_{j=1}^m\langle \phi_w-\mathcal T^{(m)}(\phi_{z^{(m)}+r\vec d}),\mathcal B_j \rangle \mathcal B_j||}\\
 &=\lim_{r\to 0}\frac{\frac{\phi_w-\mathcal T^{(m)}(\phi_{z^{(m)}+r\vec d})}{r^{l_m}}-\sum_{j=1}^m\langle \frac{\phi_w-\mathcal T^{(m)}(\phi_{z^{(m)}+r\vec d})}{r^{l_m}},\mathcal B_j \rangle \mathcal B_j}{||\frac{\phi_w -\mathcal T^{(m)}(\phi_{z^{(m)}+r\vec d})}{r^{l_m}}-\sum_{j=1}^m\langle \frac{\phi_w-\mathcal T^{(m)}(\phi_{z^{(m)}+r\vec d})}{r^{l_m}},\mathcal B_j \rangle \mathcal B_j||}\\
 &=\frac{\sum_{|\alpha|=l_m}\frac{\partial^\alpha\phi_{z}|_{z=z^{(m)}}}{\alpha!}\vec d^\alpha-\sum_{j=1}^m\langle \sum_{|\alpha|=l_m}\frac{\partial^\alpha\phi_{z}|_{z=z^{(m)}}}{\alpha!}\vec d^\alpha,\mathcal B_j \rangle \mathcal B_j}{||\sum_{|\alpha|=l_m}\frac{\partial^\alpha\phi_{z}|_{z=z^{(m)}}}{\alpha!}\vec d^\alpha-\sum_{j=1}^m\langle \sum_{|\alpha|=l_m}\frac{\partial^\alpha\phi_{z}|_{z=z^{(m)}}}{\alpha!}\vec d^\alpha,\mathcal B_j \rangle \mathcal B_j||},
 \end{align*}
 which means that the $l_m$-th directional derivative with the direction $\vec d$ of $\phi_{z}$ is involved in the G-S orthogonalization process when $r$ tends to $0$.

 Being similar with \cite{Q-2D}, we introduce $\mathcal A_k,k=1,2,...,$ the function set consisting of all possible directional derivatives of the functions in $\mathcal A_{k-1}$, where $\mathcal A_0=\{K_{\Gamma_1}(\cdot,\overline z): z\in T_{\Gamma_1}\}.$ Denote by $\mathcal A$ the set
 $$
 \mathcal A=\cup_{k=0}^\infty \mathcal A_k.
 $$
 Note that we do not make the elements in $\mathcal A_k$ normalized.
 Given a sequence of points $\{z^{(k)}\}_{k=1}^m$ in $T_{\Gamma_1},$ we can find a sequence of elements $\{\Phi_{z^{(k)}}\}_{k=1}^m$ in $\mathcal A$ such that $\{\mathcal B_k\}_{k=1}^m$ is the orthonormalization of $\{\Phi_{z^{(k)}}\}_{k=1}^m$ in the above sense. We actually make $\Phi_{z^{(k)}}=\widetilde \phi_{z^{(k)}}^{\vec d}$ with a fixed direction $\vec d.$
 Define
 \begin{align}\label{general_F}
 \widetilde {S_m}(F)=\sum_{k=1}^m \langle F,\mathcal B_k \rangle \mathcal B_k.
 \end{align}
 With the same reason as for (\ref{Sm-interpolation}) we have $$\langle \widetilde {S_m}(F), \Phi_{z^{(k)}}\rangle = \langle F, \Phi_{z^{(k)}}\rangle, \quad k=1,2,...,m.$$
 Note that $S_m(F)=\widetilde {S_m}(F)$ if all $z^{(k)}$s' are distinct from each other.
 For $\mathbb C^n,n>2$, we can similarly define $\mathcal A$ and $\widetilde {S_m}(F)$. Hereafter, we do not distinguish between $S_m(F)$ and $\widetilde {S_m}(F)$, and then we adopt the notation $S_m(F)$ for both cases.\par
\bigskip

 \begin{remark}
 	The treatment given in the above is in the spirit of the one dimensional case, which is, however, a special one in the higher dimensional case. On the other hand, due to the Cauchy-Riemann equations between the partial derivatives, there is a limited number of $h$-order partial or directional derivatives involved for each fixed $h.$
 	It can be easily computed that there exist $\binom{h+n-1}{n-1}$ linearly independent $h$-th partial derivatives. If we want to use as less as high order partial derivatives, then after at most $\sum_{h=1}^k \binom{h+n-1}{n-1}= \binom{k+n}{n}$ interactive steps one $(k+1)$-th partial derivative yet has to be involved.
 \end{remark}

 The main procedure of constructing $S_m(F)$ using an optimal $\{z^{(j)}\}_{j=1}^m$ leading to fast convergence to $F$ ($m\to \infty$) in the $H^2$-norm is as follows.
 Every time when we already have $m$ points in $T_{\Gamma_1}$,
 we select $z^{(m+1)}\in T_{\Gamma_1}$ to satisfy
 \begin{align}\label{min_pro_revise}
 \begin{split}
 z^{(m+1)}& :=\arg \max_{z\in T_{\Gamma_1}} |\langle F,\mathcal B_{m+1}^z\rangle|,
 \end{split}
 \end{align}
where $\{\mathcal B_1,...,\mathcal B_m,\mathcal B_{m+1}^z\}$ is the orthogonalization of $\{\mathcal B_1,...,\mathcal B_m,\Phi_z\}.$
Existence of such $z^{(m+1)}$ is proved in Section 4. It is, in fact, a consequence of BVC in $H^2(T_{\Gamma_1})$.
 At this moment, we assume that such $z^{(m+1)}$ exists. The convergence of $S_m(F)$ in the
 $H^2$-norm sense is given by the following result.

 \bigskip
 \begin{thm}\label{thm1}
 	For $F\in H^2(T_{\Gamma_1})$, we have
 	\begin{align}\label{convergence}
 	||F -S_m(F)||\to 0, \text{as }  m\to \infty,
 	\end{align}
 	where each element of $\{z^{(k)}\}_{k=1}^{\infty}$ is selected according to the maximal selection principle $(\ref{min_pro_revise})$.
 \end{thm}
 \begin{proof}
 	By (\ref{general_F}) and the Riesz-Fischer theorem, we know that there
 	exists $S_{\infty}(F)\in H^2(T_{\Gamma_1})$ satisfying
 	\begin{align}\label{lim1}
 	S_\infty(F)=\lim_{m\to \infty}S_m(F) \quad \text{in the
 		$H^2$-norm}.
 	\end{align}
 	If $(\ref{convergence})$ does not hold, then
 	\begin{align}
 	g=F-S_\infty(F)\not\equiv 0.
 	\end{align}
 	
 	By the Identity theorem in several complex variables analysis, we must have $b\not\in \{z^{(k)}\}_{k=1}^\infty$ such that
 	\begin{align}\label{g_b}
 	|g(b)|= \delta_0>0.
 	\end{align}
 	
 	On one hand,
 	\begin{align}\label{two term}
 	\begin{split}
 	\delta_0=|g(b)|&=| F(b)-S_\infty(F)(b)|\\
 	&\leq
 	|F(b)-S_m(F)(b)|+|S_\infty(F)(b)-S_m(F)(b)|.
 	\end{split}
 	\end{align}
 	By $(\ref{lim1})$, there exists $ N_1>0$ such that when $m>N_1$, the second term of
 	$(\ref{two term})$
 	\begin{align*}
 	|S_\infty(F)(b)-S_m(F)(b)| &=|\langle S_\infty(F)(\cdot)-S_m(F)(\cdot), K_{\Gamma_1}(\cdot, \overline b) \rangle|\\
 	&\leq ||S_\infty(F)-S_m(F)|| ||K_{\Gamma_1}(\cdot, \overline b)||\\
 	&<\frac{\delta_0}{2},
 	\end{align*}
 	where the second inequality follows from the Cauchy-Schwartz inequality.
 	Hence, we have
 	$$
 	|F(b)-S_m(F)|>\frac{\delta_0}{2}.
 	$$
 	On the other hand, by (\ref{Sm-interpolation}) we have
 	\begin{align*}
 	F(b)=S_{m+1}^b(F)(b),
 	\end{align*}
 	where $S_{m+1}^b(F)$ is defined as (\ref{general_F}), which corresponds to
 	$(z^{(1)},...,z^{(m)}, b)$.
 	By $(\ref{min_pro_revise})$, we have
 	\begin{align*}
 	||F||^2-||S_{m+1}^b(F)||^2 &=||F-S_{m+1}^b(F)||^2\\
 	&\geq ||F-S_{m+1}(F)||^2=||F||^2-||S_{m+1}(F)||^2.
 	\end{align*}
 	Therefore, there exists $N_2>0$ such that when $m>N_2$,
 	\begin{align}
 	\begin{split}
 	|F(b)-S_m(F)(b)| &= |S_{m+1}^b(F)(b)-S_m(F)(b)|\\
 	& \leq ||S_{m+1}^b(F)-S_m(F)|| ||K_{\Gamma_1}(\cdot, \overline b)||\\
 	& =||K_{\Gamma_1}(\cdot, \overline b)||(\sqrt{||S_{m+1}^b(F)||^2-||S_m(F)||^2})\\
 	& \leq ||K_{\Gamma_1}(\cdot, \overline b)||(\sqrt{||S_{m+1}(F)||^2-||S_m(F)||^2})\\
 	& =||K_{\Gamma_1}(\cdot, \overline b)||{||S_{m+1}(F)-S_m(F)||}\\
 	& <\frac{\delta_0}{2}.
 	\end{split}
 	\end{align}
 	This proves the theorem.
 \end{proof}

 Immediately, we have the following corollary.
 \begin{cor}
 	If all the conditions in Theorem \ref{thm1} are fulfilled, then, for
 	any compact subset $A$ in $T_{\Gamma_1}$,
 	$$
 	S_m(F)(z) =  \sum_{k=1}^m\langle F,\mathcal B_k\rangle \mathcal B_k(z), \quad z\in A,
 	$$
 	uniformly converges to $F(z)$ as $m\to \infty$.
 \end{cor}
 The proof of the corollary follows from the Cauchy-Schwartz inequality and the fact that the $L^2$-norm of the Cauchy-Szeg\"o kernel is uniformly bounded in any compact subset $A.$
 \medskip

 \section{The Maximal Selection Principle of POAFD}
 Denote by $\partial T_{\Gamma_1}$ the boundary of $T_{\Gamma_1}$. Since $\{z^{(1)},...,z^{(m)}\}$ are previously fixed in the procedure of selecting $z^{(m+1)}$ in (\ref{min_pro_revise}), the existence of $z^{(m+1)}$ is given by the following lemmas. For $z\neq z^{(k)},k=1,...,m,$ we recall that
 \begin{align*}
 S_{m+1}^z(F)=\sum_{k=1}^m\langle F,\mathcal B_k\rangle + \langle F,\mathcal B_{m+1}^z\rangle \mathcal B_{m+1}^z,
 \end{align*}
 where $\{\mathcal B_k\}_{k=1}^m$ is the orthogonalization of $\{\Phi_{z^{(k)}}\}_{k=1}^m,$ and $\{\mathcal B_1,...,\mathcal B_m,\mathcal B_{m+1}^z\}$ is the orthogonalization of $\{\mathcal B_1,...,\mathcal B_m,\phi_z\}$. Note that $\phi_z=K_{\Gamma_1}(\cdot,\overline z)$ and $\{\Phi_{z^{(k)}}\}\subset \mathcal A,$ where $\mathcal A$ contains all higher order directional derivatives of $K_{\Gamma_1}(w,\overline z)$ with respect to $z$.

 First, we show that
 \begin{lem}\label{re-ex-lem}
 	Suppose that $F\in H^2(T_{\Gamma_1})$ and $z^{(j)}\in T_{\Gamma_1},
 	j=1,...,m,$ are fixed.
 	 If
 	\begin{align}\label{g-ex-cond}
 	\begin{split}
 	\lim_{z \to \beta} \frac{|F(z)|}{||\phi_z||}
 	&= 0,\\
 	\lim_{z\to \beta}\frac{|\Phi_{z^{(j)}}(z)|}{||\phi_z||}&=0,\quad j=1,2,...,m,
 	\end{split}
 	\end{align}
 	where $\beta\in \partial T_{\Gamma_1}$, then
 	\begin{align}\label{g-ex-boundary}
 	\lim_{z\to \beta}||F- S_{m+1}^z(F)||=
 	||F-S_{m}(F)||,
 	\end{align}
 	and if
 	\begin{align}\label{g-ex-cond1}
 	\begin{split}
 	\lim_{|z| \to \infty}
 	\frac{|F(z)|}{||\phi_z||} &= 0,\\
 	\lim_{|z|\to \infty}\frac{|\Phi_{z^{(j)}}(z)|}{||\phi_z||}&=0,\quad j=1,2,...,m,
 	\end{split}
 	\end{align}
 	then
 	\begin{align}\label{g-ex-infinity}
 	\lim_{|z| \to \infty}||F- S_{m+1}^z(F)||=
 	||F-S_{m}(F)||.
 	\end{align}
 \end{lem}

 \begin{proof}
 	We adopt the notation given in (\ref{G-S}). Based on (\ref{general_F}), we have
 	$$
 	\Vert F-S_{m+1}^z(F) \Vert^2=\Vert F-S_{m}(F)\Vert^2-|\langle F, \mathcal B_{m+1}^z \rangle|^2.
 	$$
 	To get (\ref{g-ex-boundary}), we need to show $|\langle F,\mathcal B_{m+1}^z \rangle|\to 0$ as $z\to \beta$.
 	In fact, by the Gram-Schmidt orthogonalization process,
 	\begin{align}
 	\begin{split}
 	|\langle F,\mathcal B_{m+1}^z\rangle|
 	& = \frac{|\langle F, \phi_{z} - \sum_{k=1}^{m}\langle \phi_z, \mathcal B_k\rangle\mathcal B_k\rangle|}{\|\phi_z - \sum_{k=1}^{m}\langle \phi_z, \mathcal B_k\rangle\mathcal B_k\|}\\
 	& = \frac{|\langle F, \frac{\phi_z}{\|\phi_z\|} - \sum_{k=1}^{m}\langle \frac{\phi_z}{\|\phi_z\|}, \mathcal B_k\rangle\mathcal B_k\rangle|}{\sqrt{1 - \sum_{k=1}^{m}|\langle \frac{\phi_z}{\|\phi_z\|}, \mathcal B_k\rangle|^2}}.
 	\end{split}
 	\end{align}
 	Note that each $\mathcal B_k$ is a linear combination of $\{\Phi_{z^{(j)}}\}_{j=1}^k$ as given in Section 3. Then, by (\ref{g-ex-cond}) we can have (\ref{g-ex-boundary}). We can also conclude (\ref{g-ex-infinity}) in a similar way. The proof is complete.
 \end{proof}

Note the conditions (\ref{g-ex-cond}) and (\ref{g-ex-cond1}) essentially follow from
 \begin{align}\label{weak-bvc-tube}
 \begin{split}
 \lim_{z \to \beta\in \partial T_{\Gamma_1}} \frac{|F(z)|}{||\phi_z||} &=0\\
 \lim_{|z| \to \infty}
 \frac{|F(z)|}{||\phi_z||} &= 0,
 \end{split}
 \end{align}
which is called the {\it boundary vanishing condition} (BVC).

In fact, we will prove a strong version of BVC in $H^2(T_{\Gamma_1})$.
 Define
\begin{align}\label{D-element}
\begin{split}
\phi_{\alpha,z}=\frac{\partial^{|\alpha|}K_{\Gamma_1}(\cdot, \overline
	{z})}{\partial{\overline z_1}^{\alpha_1}\partial{\overline z_2}^{\alpha_2}\cdots
	\partial{\overline z_n}^{\alpha_n}} =\left(\frac{-1}{2\pi i}\right)^n \prod_{j=1}^n
\frac{\alpha_j !}{(w_j-\overline z_j)^{\alpha_j+1}},
\end{split}
\end{align}
where all elements of $n$-tuple $\alpha=(\alpha_1,...,\alpha_n)$ are non-negative integers and $|\alpha|=\sum_{j=1}^n \alpha_j \geq 0$. In particular, $\phi_{0,z}=\phi_z=K_{\Gamma_1}(\cdot,\overline z).$

 \begin{lem}\label{4.2}
 	For $1<p<\infty$, $z=x+iy\in T_{\Gamma_1}$ and $\alpha=(\alpha_1,...,\alpha_n)$,
 	$$\int_{\mathbb R^n} \left| \prod_{j=1}^n
 	\frac{1}{(\xi_j-\overline z_j)^{\alpha_j+1}} \right|^p d\xi_1 \cdots d\xi_n= \pi^{\frac{n}{2}}\prod_{j=1}^n \frac{\Gamma(-\frac{1}{2}+\frac{p(\alpha_j+1)}{2})}{\Gamma(\frac{p(\alpha_j+1)}{2})}\left(\frac{1}{y_j}\right)^{p(\alpha_j+1)-1}.$$
 \end{lem}

 \begin{proof}
 	\begin{align*}
 	\int_{\mathbb R^n}\left|  \prod_{j=1}^n
 	\frac{1}{(\xi_j-\overline z_j)^{\alpha_j+1}} \right|^p d\xi_1 \cdots d\xi_n
 	&= \prod_{j=1}^n\int_{-\infty}^\infty  \left|
 	\frac{1}{(|\xi_j-x_j|^2+|y_j|^2)} \right|^{\frac{p({\alpha_j+1})}{2}} d\xi_j\\
 	&= \prod_{j=1}^n \left(\frac{1}{y_j}\right)^{p(\alpha_j+1)-1}\int_{-\infty}^\infty  \left|
 	\frac{1}{t_j^2+1} \right|^{\frac{p({\alpha_j+1})}{2}} dt_j\\
 	&= \pi^{\frac{n}{2}}\prod_{j=1}^n \frac{\Gamma(-\frac{1}{2}+\frac{p(\alpha_j+1)}{2})}{\Gamma(\frac{p(\alpha_j+1)}{2})}\left(\frac{1}{y_j}\right)^{p(\alpha_j+1)-1},
 	\end{align*}
 	where the third equality is due to changing variables by setting $t_j=\frac{\xi_j-x_j}{y_j}$, and $\Gamma(\cdot)$ is the Gamma function.
 \end{proof}

 Consequently, we have
 \begin{align}\label{phi-norm}
 \begin{split}
 ||\phi_{\alpha,z}||^2 &=\left(\frac{1}{2\pi}\right)^n\left(\prod_{j=1}^n (\alpha_j!)^2\right) (\pi)^{\frac{n}{2}}\prod_{j=1}^n \frac{\Gamma(\alpha_j+\frac{1}{2})}{\Gamma(\alpha_j+1)}\left(\frac{1}{y_j}\right)^{2\alpha_j+1}\\
 & = \prod_{j=1}^n\frac{(2\alpha_j)!}{(2y_j)^{2\alpha_j+1}}.
 \end{split}
 \end{align}

 BVC in $H^2(T_{\Gamma_1})$ is a consequence of the following lemmas. In fact,  by Lemma \ref{re-ex-lem} and Lemmas \ref{first-app1} - \ref{third-app1}, through a compact argument we conclude that $z^{(m+1)}$ in (\ref{min_pro_revise}) can be achieved in a compact subset of $T_{\Gamma_1}$.
 In addition, Lemmas \ref{first-app1}-\ref{third-app1} can be regarded as the Riemann-Lebesgue Lemma for the Cauchy-Szeg\"o kernel.
 \begin{lem}\label{first-app1}
 	For $F\in H^2(T_{\Gamma_1})$, and $\alpha=(\alpha_1,...,\alpha_n)$,
 	\begin{align}\label{app1}
 	\lim_{y\in \Gamma_1, y\to \beta}\frac{|\langle F, \phi_{\alpha, z} \rangle|}{||\phi_{\alpha,z}||}=0, \quad z=x+iy\in T_{\Gamma_1},
 	\end{align}
 	holds uniformly for $x\in\mathbb R^n$, where $\beta\in \partial \Gamma_1$.
 \end{lem}
 \begin{proof}
 	Since
 	$$\langle F, \phi_{\alpha, z} \rangle=\int_{\mathbb R^n} F(\xi)\overline {\phi_{\alpha, z}(\xi)}d\xi,$$ and $F(\xi)$ is the limit of $F(\xi+i\eta)$ in the $L^2$-norm, we can find $G(\xi)\in L^2(\mathbb R^n)\cap L^p(\mathbb R^n), 2< p<\infty,$ such that for any $\epsilon>0$, $||F-G||_{L^2(\mathbb R^n)}<\frac{\epsilon}{2}$.
 	We also have
 	\begin{align}
 	\begin{split}
 	\frac{|\langle F, \phi_{\alpha, z} \rangle|}{||\phi_{\alpha, z}||}&\leq \frac{|\langle F-G,\phi_{\alpha,z} \rangle|}{||\phi_{\alpha,z}||} + \frac{|\langle G, \phi_{\alpha,z}\rangle| }{||\phi_{\alpha,z}||}\\
 	& \leq \frac{||F-G||~||\phi_{\alpha,z}||}{||\phi_{\alpha,z}||} +\frac{|\langle G, \phi_{\alpha,z}\rangle| }{||\phi_{\alpha,z}||}\\
 	& \leq \frac{\epsilon}{2} + \frac{|\langle G, \phi_{\alpha,z}\rangle| }{||\phi_{\alpha,z}||}.
 	\end{split}
 	\end{align}
 	It suffices to prove that, for $y\in \Gamma_1, y\to \beta$,
 	\begin{align}\label{app1-fact1}
 	\frac{\int_{\mathbb R^n}|G(\xi)\overline{\phi_{\alpha, z}(\xi)}|d\xi}{||\phi_{\alpha,z}||}<\frac{\epsilon}{2}.
 	\end{align}
 	Indeed, by applying H\"older's inequality to $\int_{\mathbb R^n}|G(\xi)\overline{\phi_{\alpha, z}(\xi)}|d\xi$, (\ref{app1-fact1}) follows from
 	\begin{align*}
 	\lim_{y\in \Gamma_1,y\to \beta}\frac{(\int_{\mathbb R^n}|\phi_{\alpha, z}|^q d\xi)^{\frac{1}{q}}}{(\int_{\mathbb R^n}|\phi_{\alpha, z}|^2 d\xi)^{\frac{1}{2}}}&=\lim_{y\in \Gamma_1,y\to \beta}\frac{\pi^{\frac{n}{2q}}\prod_{j=1}^n \left(\frac{\Gamma(-\frac{1}{2}+\frac{q(\alpha_j+1)}{2})}{\Gamma(\frac{q(\alpha_j+1)}{2})}\right)^{\frac{1}{q}}\left(\frac{1}{y_j}\right)^{(\alpha_j+1)-\frac{1}{q}}}{\pi^{\frac{n}{4}}\prod_{j=1}^n \left(\frac{\Gamma(\alpha_j+\frac{1}{2})}{\Gamma(\alpha_j+1)}\right)^{\frac{1}{2}}\left(\frac{1}{y_j}\right)^{\alpha_j+\frac{1}{2}}}\\
 	&=C\lim_{y\in \Gamma_1,y\to \beta}\prod_{j=1}^n y_j^{\frac{1}{2}-\frac{1}{p}}\\
 	&=0,
 	\end{align*}
 	where $C$ is a constant, and $q$ satisfies $\frac{1}{p}+\frac{1}{q}=1.$ The last equality follows from $p>2$ and the fact that there exists some $y_j\to 0$ when $y\to \beta$.
 \end{proof}

 \begin{lem}\label{second-app1}
 	For $F\in H^2(T_{\Gamma_1})$, and $\alpha=(\alpha_1,...,\alpha_n)$,
 	\begin{align}\label{app1-fact2}
 	\lim_{y\in \Gamma_1, |y|\to \infty} \frac{|\langle F, \phi_{\alpha, z} \rangle|}{||\phi_{\alpha,z}||}=0, \quad z=x+iy\in T_{\Gamma_1},
 	\end{align}
 	holds uniformly for $x\in \mathbb R^n$.
 \end{lem}
 \begin{proof}
 	By Theorem \ref{rf_p}, for any $\epsilon>0$, we can find $y^\prime\in \Gamma_1$ such that
 	$$
 	\int_{\mathbb R^n}|F(\xi)-F(\xi+iy^\prime)|^2 d\xi<\epsilon.
 	$$
 	\begin{align}\label{app1-fact2-ineq}
 	\begin{split}
 	\frac{|\langle F, \phi_{\alpha, z} \rangle|}{||\phi_{\alpha, z}||}&\leq \frac{|\langle F(\cdot)-F(\cdot+iy^\prime),\phi_{\alpha,z} \rangle|}{||\phi_{\alpha,z}||} + \frac{|\langle F(\cdot+iy^\prime), \phi_{\alpha,z}\rangle| }{||\phi_{\alpha,z}||}\\
 	& \leq \frac{||F(\cdot)-F(\cdot+iy^\prime)||~||\phi_{\alpha,z}||}{||\phi_{\alpha,z}||} +\frac{|\langle F, \phi_{\alpha,z+iy^\prime}\rangle| }{||\phi_{\alpha,z}||}\\
 	& \leq \epsilon + \frac{|\langle F, \phi_{\alpha,z+iy^\prime}\rangle| }{||\phi_{\alpha,z}||}.
 	\end{split}
 	\end{align}
 	Note that we can find $G\in L^2(\mathbb R^n)\cap L^p(\mathbb R^n), 1<p<2$, such that, for any $\epsilon>0$, $||F-G||_{L^2(\mathbb R^n)}<\epsilon$.
 	By applying the argument in Lemma \ref{first-app1}, we can easily show that, for $y\in \Gamma_1$ and $|y|$ large enough, $$\frac{|\langle F, \phi_{\alpha,z+iy^\prime}\rangle| }{||\phi_{\alpha,z}||}<C\epsilon,$$ where $C$ is a constant.
 \end{proof}

 \begin{lem}\label{third-app1}
 	For $F\in H^2(T_{\Gamma_1})$, and $\alpha=(\alpha_1,...,\alpha_n)$,
 	\begin{align}\label{app1-fact3}
 	\lim_{|x|\to \infty} \frac{|\langle F, \phi_{\alpha, z} \rangle|}{||\phi_{\alpha,z}||}=0, \quad z=x+iy\in T_{\Gamma_1},
 	\end{align}
 	holds uniformly for $y\in \Gamma_1$.
 \end{lem}
 \begin{proof}
 	By Lemmas \ref{first-app1} and \ref{second-app1}, it suffices to prove that
 	\begin{align}\label{g-cone-fact2}
 	\lim_{|x|\to \infty} \frac{|\langle F, \phi_{\alpha, z} \rangle|}{||\phi_{\alpha,z}||}=0
 	\end{align}
 	holds uniformly for $y\in A_0$, where $A_0$ is a compact subset in $\Gamma_1$.

 	Since $\overline{ {\text span}} \{K_{\Gamma_1}(\cdot,\overline z), z\in T_{\Gamma_1}\}=H^2(T_{\Gamma_1})$, we have $\{w^{(j)}\}_{j=1}^N$ in $T_{\Gamma_1}$ such that
 	$$\|F-G_N\|<\frac{\epsilon}{2},$$
 	where $G_N=\sum_{j=1}^Nc_j K_{\Gamma_1}(\cdot, \overline {w^{(j)}})\in H^2(T_{\Gamma_1}).$
 	Hence, we have
 	$$
 	\frac{|\langle F, \phi_{\alpha, z} \rangle|}{||\phi_{\alpha,z}||}\leq \frac{|\langle F-G_N, \phi_{\alpha, z} \rangle|}{||\phi_{\alpha,z}||}+ \frac{|\langle G_N, \phi_{\alpha, z} \rangle|}{||\phi_{\alpha,z}||}<\frac{\epsilon}{2}+ \frac{|\langle G_N, \phi_{\alpha, z} \rangle|}{||\phi_{\alpha,z}||}.
 	$$
 	It suffices to show that, for a fixed $w=\xi+i\eta\in T_{\Gamma_1}$, when $|x|$ is large enough,
 	\begin{align*}
 	\frac{|\langle K_{\Gamma_1}(\cdot,\overline w), \phi_{\alpha,z}\rangle|}{||\phi_{\alpha,z}||}&=\prod_{j=1}^n \frac{(2y_j)^{\alpha_j+\frac{1}{2}}\alpha_j!}{|z_j-\overline w_j|^{\alpha_j+1}\sqrt{(2\alpha_j)!}}\\
 	&= \prod_{j=1}^n \frac{(2y_j)^{\alpha_j+\frac{1}{2}}\alpha_j!}{|(x_j-\xi_j)^2+(y_j+\eta_j)^2|^{\frac{\alpha_j+1}{2}}\sqrt{(2\alpha_j)!}}\\
 	& <\frac{\epsilon}{2}.
 	\end{align*}
 	The last inequality is based on the fact that there exists $x_j$ satisfying that $|x_j|\to\infty$ as $|x|\to \infty$.
 \end{proof}

 Note that we can alternatively prove Lemma \ref{first-app1} and Lemma \ref{second-app1} by using the density argument as used in proving Lemma \ref{third-app1}.
Since the tube $T_{\Gamma_1}$ is very special, in Appendix A, we will prove one more boundary vanishing property of $\phi_{\alpha,z}.$ Combining it with Lemmas \ref{first-app1} - \ref{third-app1}, we can give another rational approximation in $H^2(T_{\Gamma_1})$ that is analogous to the one in \cite{WQ1}.

 \section{Further Results}
 In this section we consider the convergent rate aspect of POAFD in $H^2(T_{\Gamma_1})$, and give rational approximation of functions in $L^2(\mathbb R^n)$ by POAFD in $H^2(T_{\Gamma_1})$.

 \subsection{Rate of convergence}
 As in \cite{DT}, we first introduce the function class
 $$
 H^2(T_{\Gamma_1},M)= \left \{F\in H^2(T_{\Gamma_1}): F=\sum_{j=1}^\infty c_j\frac{\phi_{w^{(j)}}}{||\phi_{w^{(j)}}||}, w^{(j)}\in T_{\Gamma_1}, \sum_{j=1}^\infty |c_j|\leq M \right \},
 $$
 where $\phi_{z}=K_{\Gamma_1}(\cdot,\overline z).$
 We give the convergent rate of the AFD-type approximation of functions in $H^2(T_{\Gamma_1},M)$.
 The result is stated as follows.
 \begin{thm}\label{con-rate}
 	For $F\in H^2(T_{\Gamma_1},M)$, and $S_m(F)$ corresponding to the sequence $\{z^{(k)}\}_{k=1}^m$, where each element of $\{z^{(k)}\}_{k=1}^m$ is selected according to the maximal selection principle $(\ref{min_pro_revise})$, we have
 	$$
 	\| F-S_{m}(F) \|\leq \frac{M}{\sqrt{m+1}}.
 	$$
 \end{thm}

 To prove Theorem \ref{con-rate}, we need the following result.
 \begin{lem}[\cite{DT}]\label{DT-lem}
 	Let $\{d_k\}_{k=1}^\infty$ be a sequence of nonnegative numbers satisfying
 	$$
 	d_1\leq A, \quad d_{k+1}\leq d_k\left(1-\frac{d_k}{A}\right).
 	$$
 	Then there holds
 	$$
 	d_k\leq \frac{A}{k}.
 	$$
 \end{lem}
 {\em Proof of Theorem \ref{con-rate}:}\\
 For $F\in H^2(T_{\Gamma_1}, M)$, we have $F=\sum_{k=1}^\infty c_k\frac{\phi_{w^{(k)}}}{\|\phi_{w^{(k)}}\|}$ and
 $$
 ||F||\leq \sum_{j=1}^\infty |c_k|\leq M.
 $$
 By (\ref{general_F}), we have
 $$
 \|S_{m}(F)\|^2 =\sum_{k=1}^{m} |\langle F, \mathcal B_k \rangle|^2,
 $$
 and
 \begin{align}\label{F_m}
 \|F_{m+1}\|^2= \|F_m\|^2-|\langle F, \mathcal B_m \rangle|^2=\|F_m\|^2-|\langle F_m, \mathcal B_m \rangle|^2,
 \end{align}
 where $F_{m+1}=F-S_{m}(F)$ with $F_1=F$.
 By (\ref{G-S}),
 \begin{align}\label{F_mB_m}
 \begin{split}
 |\langle F_m,\mathcal B_m\rangle| &= \frac{|\langle F_m, {\gamma_m}\rangle|}{\|\gamma_m\|}\\
 & = \frac{|\langle F_m, \Phi_{z^{(m)}} - \sum_{k=1}^{m-1}\langle \Phi_{z^{(m)}}, \mathcal B_k\rangle\mathcal B_k\rangle|}{\|\Phi_{z^{(m)}} - \sum_{k=1}^{m-1}\langle \Phi_{z^{(m)}}, \mathcal B_k\rangle\mathcal B_k\|}\\
 & = \frac{|\langle F_m, \frac{\Phi_{z^{(m)}}}{\|\Phi_{z^{(m)}}\|}\rangle|}{\|\frac{\Phi_{z^{(m)}}}{\|\Phi_{z^{(m)}}\|} - \sum_{k=1}^{m-1}\langle \frac{\Phi_{z^{(m)}}}{\|\Phi_{z^{(m)}}\|}, \mathcal B_k\rangle\mathcal B_k\|}\\
 &\geq \frac{|\langle F_m, \frac{\phi_{z^{(m)}}}{\|\phi_{z^{(m)}}\|}\rangle|}{\|\frac{\phi_{z^{(m)}}}{\|\phi_{z^{(m)}}\|} - \sum_{k=1}^{m-1}\langle \frac{\phi_{z^{(m)}}}{\|\phi_{z^{(m)}}\|}, \mathcal B_k\rangle\mathcal B_k\|}\\
 &\geq |\langle F_m, \frac{\phi_{z^{(m)}}}{\|\phi_{z^{(m)}}\|}\rangle|,
 \end{split}
 \end{align}
 where the last inequality is based on $$\|\frac{\phi_{z^{(m)}}}{\|\phi_{z^{(m)}}\|} - \sum_{k=1}^{m-1}\langle \frac{\phi_{z^{(m)}}}{\|\phi_{z^{(m)}}\|}, \mathcal B_k\rangle\mathcal B_k\|^2 = 1-\sum_{k=1}^{m-1}|\langle \frac{\phi_{z^{(m)}}}{\|\phi_{z^{(m)}}\|},\mathcal B_k\rangle|^2\leq 1.$$
 Combining (\ref{min_pro_revise}), (\ref{F_m}) and (\ref{F_mB_m}), we have
 \begin{align}\label{F_mB_m1}
 \begin{split}
 |\langle F_m, \mathcal B_m \rangle| &= \sup_{z\in T_{\Gamma_1}}|\langle F_m, \mathcal B_{\{z^{(1)},z^{(2)},...,z^{(m-1)},z\}} \rangle|\\
 &\geq \sup_{z\in T_{\Gamma_1}}|\langle F_m, \frac{\phi_{z}}{\|\phi_{z}\|} \rangle|\\
 &\geq \sup_{z\in \{w^{(k)}\}_{k=1}^\infty}|\langle F_m, \frac{\phi_{z}}{\|\phi_{z}\|} \rangle|.\\
 \end{split}
 \end{align}
 Notice that
 \begin{align*}
 \|F_m\|^2&=|\langle F_m, F\rangle|
 =|\langle F_m, \sum_{k=1}^\infty c_k\frac{\phi_{w^{(k)}}}{\|\phi_{w^{(k)}}\|}\rangle|
 \leq M\sup_{z\in \{w^{(k)}\}_{k=1}^\infty} |\langle F_m, \frac{\phi_{z}}{\|\phi_{z}\|}\rangle|.
 \end{align*}
 Hence,
 \begin{align*}
 \|F_{m+1}\|^2&=\|F_m\|^2-|\langle F_m, \mathcal B_m\rangle|^2\\
 &\leq \|F_m\|^2-\sup_{z\in \{w^{(k)}\}_{k=1}^\infty}|\langle F_m, \frac{\phi_{z}}{\|\phi_{z}\|} \rangle|^2\\
 &\leq \|F_m\|^2-\frac{\|F_m\|^4}{M^2}\\
 &=\|F_m\|^2\left(1-\frac{\|F_m\|^2}{M}\right).
 \end{align*}
 By Lemma \ref{DT-lem}, we conclude the desired result.
 \quad \hfill$\Box$\vspace{2ex}

 \subsection{Rational approximation of functions in $L^2(\mathbb R^n)$}
 It is known that, for $f\in L^2(\mathbb R)$, one can have $f=f^+ + f^-$,  where $f^+$ and $f^-$ are non-tangential boundary limits of functions contained in $H^2(\mathbb C_+)$ and $H^2(\mathbb C_-)$, respectively. Then, rational approximation of functions in $L^2(\mathbb R)$ can be easily obtained by rational approximations of functions in $H^2(\mathbb C_+)$ and $H^2(\mathbb C_-)$. Here we give rational approximation of functions in $L^2(\mathbb R^n)$ in a similar manner.
 Define $\sigma_j=(\sigma_j(1),\sigma_j(2),...,\sigma_j(n)),1\leq j\leq 2^n$, whose elements are $+$ and $-$,
 and $$\Gamma_{\sigma_j}=\{y\in \mathbb R^n;y_k>0 \text{ if } \sigma_j(k)=+ \text{ and } y_k<0 \text{ if } \sigma_j(k)=-, j=1,2,...,n\}.$$
 Observe that
 $\mathbb R^n=\cup_{j=1}^{2^n}\overline {\Gamma_{\sigma_j}}. $
 For $F\in L^2(\mathbb R^n)$, the following result is known.
 \begin{thm}[\cite{SW,P}]\label{Hardy-dec-pro}
 	For $F\in L^2(\mathbb R^n),$ if
 	\begin{align}\label{Hardy-projection}
 	\begin{split}
 	F_{\sigma_j}(z)=\int_{\mathbb R^n}F(\xi)\overline{K_{\Gamma_{\sigma_j}}(\xi, \overline z)}
 	d\xi=\frac{(-1)^{m_j}}{(2\pi i)^n}\int_{\mathbb R^n}F(\xi)\prod_{k=1}^n
 	\frac{1}{\xi_k-z_k}d\xi_1\cdots d\xi_n, \quad
 	\end{split}
 	\end{align}
 	where $z\in T_{\Gamma_{\sigma_j}}$ and $m_j$ denotes the number of minus signs in $\sigma_j$,
 	then $F_{\sigma_j}(z)$ is holomorphic on $T_{\Gamma_{\sigma_j}}$, and for $F_{\sigma_j}(x+iy)$ as a function of $x$,
 	\begin{align}\label{Riesz-inequality}
 	\Vert F_{\sigma_j}(\cdot+iy) \Vert_{L^2(\mathbb R^n)}\leq C \Vert F\Vert_{L^2(\mathbb R^n)},
 	\end{align}
 	where $C$ is a constant that is independent of $F$ and $y$.\\
 	Furthermore,
 	\begin{align}\label{Hardy-decomposition}
 	F(x)=\sum_{j=1}^{2^n}F_{\sigma_j}(x),\quad x\in \mathbb R^n, \text{ in the } L^2 \text{-sense,}
 	\end{align}
 	where $F_{\sigma_j}(x)=\lim_{y\in \Gamma_{\sigma_j},y\to 0}F_{\sigma_j}(x+iy)$ is the limit function in the $L^2$-norm.
 \end{thm}
 \begin{remark}
 	It is noted that Theorem \ref{Hardy-dec-pro} is a summary of partial results given in \cite{SW,P}. In fact, the conclusion given in Theorem \ref{Hardy-dec-pro} holds also for $F\in L^p(\mathbb R^n), 1<p<\infty$ (see \cite{P}). For more information, we also refer to, e.g. \cite{Ti,Till,Vl}. Based on the definition of Hardy spaces, the inequality (\ref{Riesz-inequality}) implies that $F_{\sigma_j}\in H^2(T_{\Gamma_{\sigma_j}})$. The formula (\ref{Hardy-projection}) is called the Hardy projection of $F$. Hence the formula (\ref{Hardy-decomposition}) means that $F$ can be decomposed into a sum of boundary limit functions of functions in the Hardy spaces on tubes over octants. Moreover, for $F\in L^p(\mathbb R^n), 1<p<\infty$, $F_{\sigma_j}(x)$ can be characterized as the Fourier transform of a function supported on $\Gamma_{\sigma_j}$ in the distribution sense (see \cite{P}). This can be regard as a generalization of the Paley-Wiener theorem (see Theorem \ref{PW}).
 	Recently, the analogues of the Paley-Wiener theorem for $H^p(T_\Gamma)$,$1\leq p\leq \infty$ (for $2<p\leq \infty$ in the distribution sense) are given in \cite{Li-Deng-Qian}, where $\Gamma$ is a regular cone.
 \end{remark}

 Due to Theorem \ref{Hardy-dec-pro} and the above discussion, we can reduce the relevant study to $F_{\sigma_j}(1\leq j\leq 2^n$) when considering the problem of rational approximation of $F\in L^2(\mathbb R^n)$. Therefore, for each $F_{\sigma_j}$ we can obtain an approximation of $F_{\sigma_j}$ given by POAFD.
 Moreover, for a real-valued function $F$ we only need to deal with the related $2^{n-1}$ Hardy spaces.
 For instance, we interpret this in $\mathbb C^2$. If $\{z^{(k)}=(z^{(k)}_1,z^{(k)}_2)\}_{k=1}^\infty$ is a sequence making $S_{m}^{+,+}(F)\to F_{+,+}$ as $m\to \infty $ in $H^2(T_{\Gamma_{+,+}})$, then we have that $\{(\overline {z^{(k)}_1},\overline {z^{(k)}_2})\}_{k=1}^\infty$ is a sequence making $S_{m}^{-,-}(F)\to F_{-,-}$ as $m\to \infty$ in $H^2(T_{\Gamma_{-,-}})$. The same argument holds for the cases $F_{+,-}$ and $F_{-,+}$.

\bigskip
\begin{appendix}

\section{A Non-Orthogonal Expansion with Simpler Algorithm}
\setcounter{equation}{0}
\renewcommand\theequation{A.\arabic{equation}}
By using Lemmas \ref{first-app1} - \ref{third-app1} and the following lemma (Lemma \ref{MP-first-case}), we can give another kind of rational approximation to functions in $H^2(T_{\Gamma_1})$. Such rational approximation is analogous to the one given in \cite{WQ1}. Specifically, we apply the idea of greedy algorithm to $H^2(T_{\Gamma_1})$ with the dictionary
\begin{align}\label{dictionary}
\mathcal D=\left
\{\psi_{\alpha,z}(w)=\frac{\phi_{\alpha,z}(w)}{||\phi_{\alpha,z}||};
|\alpha|=\sum_{j=1}^n \alpha_j \geq 0, z, w\in T_{\Gamma_1} \right
\}.
\end{align}
Recall that $\phi_{\alpha,z}(w)$ is defined as
\begin{align*}
\begin{split}
\phi_{\alpha,z}=\frac{\partial^{|\alpha|}K_{\Gamma_1}(\cdot, \overline
	{z})}{\partial{\overline z_1}^{\alpha_1}\partial{\overline z_2}^{\alpha_2}\cdots
	\partial{\overline z_n}^{\alpha_n}} =\left(\frac{-1}{2\pi i}\right)^n \prod_{j=1}^n
\frac{\alpha_j !}{(w_j-\overline z_j)^{\alpha_j+1}},
\end{split}
\end{align*}
where all elements of $n$-tuple $\alpha=(\alpha_1,...,\alpha_n)$ are non-negative integers and $|\alpha|=\sum_{j=1}^n \alpha_j \geq 0$. 

We briefly give an introduction to greedy algorithm with the dictionary $\mathcal D$.
Let $F\in H^2(T_{\Gamma_1})$. Then, by greedy
algorithm, one can have
\begin{align}\label{mp}
F=\sum_{l=1}^m \langle R^l F, \psi_{\alpha^{(l)},z^{(l)}}\rangle
\psi_{\alpha^{(l)},z^{(l)}}  + R^{m+1} F,
\end{align}
where $R^l F$ is defined by
$$
R^0F=F, R^lF=R^{l-1}F- \langle R^{l-1} F,
\psi_{\alpha^{(l)},z^{(l)}}\rangle \psi_{|\alpha^{(l)}|,z^{(l)}}, \quad l\geq 1,
$$
and $\psi_{\alpha^{(l)},z^{(l)}}$ satisfies
\begin{align}\label{max-cond}
|\langle R^{l}F, \psi_{\alpha^{(l)},z^{(l)}} \rangle|=
\sup_{\psi_{\alpha,z}\in \mathcal D}{ |\langle R^{l}F,
	\psi_{\alpha,z} \rangle|}.
\end{align}

In the following discussion, we focus on the existence of $\psi_{\alpha^{(l)},z^{(l)}}$ in (\ref{max-cond}) for each $l\geq 1$.
To this end, we show the following result by using the technique in \cite[Lemma 4.8]{WQ1}.
\begin{lem}\label{MP-first-case}
	For $F\in H^2(T_{\Gamma_1})$,
	\begin{align}\label{MP-case1}
	\lim_{|\alpha|\to \infty}\frac{|\langle F, \phi_{\alpha, z} \rangle|}{||\phi_{\alpha,z}||}=0
	\end{align}
	holds uniformly for $z\in T_{\Gamma_1}$.
\end{lem}
\begin{proof}
	As shown in Lemma \ref{third-app1}, for any $\epsilon>0$, there exists \\$G_N=\sum_{j=1}^Nc_j K_{\Gamma_1}(\cdot, \overline {w^{(j)}})$ such that
	$$
	||F-G_N||<\frac{\epsilon}{2}.
	$$
	Therefore,
	we only need to prove that for any fixed $w=\xi+i\eta \in T_{\Gamma_1}$
	$$
	\lim_{|\alpha|\to \infty}\frac{|\langle K_{\Gamma_1}(\cdot, \overline w), \phi_{\alpha, z} \rangle|}{||\phi_{\alpha,z}||}=0.
	$$
	In fact, we have
	\begin{align}\label{MP-case1-1}
	\begin{split}
	\frac{|\langle K_{\Gamma_1}(\cdot, \overline w), \phi_{\alpha, z} \rangle|}{||\phi_{\alpha,z}||}&=\prod_{j=1}^n \frac{(2y_j)^{\alpha_j+\frac{1}{2}}\alpha_j!}{|z_j-\overline w_j|^{\alpha_j+1}\sqrt{(2\alpha_j)!}}\\
	&\leq \prod_{j=1}^n \frac{(2y_j)^{\alpha_j+\frac{1}{2}}\alpha_j!}{(y_j+\eta_j)^{\alpha_j+1}\sqrt{(2\alpha_j)!}}\\
	&\leq \prod_{j=1}^n \frac{((\alpha_j+\frac{1}{2})\eta_j)^{\alpha_j+\frac{1}{2}}2^{2\alpha_j+1}\alpha_j!}{((\alpha_j+1)\eta_j)^{\alpha_j+1}2^{\alpha_j+1}\sqrt{(2\alpha_j)!}}\\
	&= \prod_{j=1}^n \frac{(\alpha_j+\frac{1}{2})^{\alpha_j+\frac{1}{2}}2^{\alpha_j}\alpha_j!}{(\alpha_j+1)^{\alpha_j+1}\sqrt{\eta_j}\sqrt{(2\alpha_j)!}}\\
	&\leq \prod_{j=1}^n C_j \eta_j^{-\frac{1}{2}} \alpha_j^{-\frac{1}{4}},
	\end{split}
	\end{align}
	where $C_j$ is a constant that is independent of $\alpha_j$. The second inequality is based on the fact that $\frac{(2y_j)^{\alpha_j+\frac{1}{2}}}{(y_j+\eta_j)^{\alpha_j+1}}$ attains its maximum value at $y_j=2(\alpha_j+\frac{1}{2})\eta_j$ for $ 1\leq j\leq n$, and the last inequality follows from the Stirling's formula
	$$\Gamma(h+1)\sim h^{h+\frac{1}{2}} e^{-h}\sqrt{2\pi}, \quad h\in \mathbb R, h\to \infty.$$
	The proof is complete.
\end{proof}

Combining Lemma \ref{MP-first-case} and Lemmas \ref{first-app1} - \ref{third-app1}, we conclude the existence of the optimal $\psi_{\alpha^{(l)},z^{(l)}},l\geq 1.$
Note that
$$
||F||^2=\sum_{l=1}^m |\langle R^l F,
\psi_{\alpha^{(l)},z^{(l)}}\rangle|^2 + ||R^{m+1}F||^2,
$$
although $\{\psi_{\alpha^{(l)},z^{(l)}}, l=1,2...,m\}$ is may not be an orthogonal system. Based on Theorem 1 in \cite{MZ} and the fact that $\overline {{\text span}} \mathcal D=H^2(T_{\Gamma_1})$, we have
\begin{align}\label{convergence-MP}
\lim_{m\to \infty}||R^mF||=0.
\end{align}
In a general reproducing kernel Hilbert space an ordinary greedy algorithm (GA) scheme is usually not as effective as what is called orthogonal greedy algorithm (OGA) by the construction, and the latter not as effective as POAFD. But with the extended dictionary in (\ref{dictionary}) incorporating all possible partial derivatives of the reproducing kernels OGA is comparable with POAFD with respect to only the reproducing kernels.

For further discussion on such approximation and greedy algorithm, see e.g. \cite{WQ1,DT,MZ,Te}.

 \section{Results on Regular Domains}
 \setcounter{equation}{0}
 \renewcommand\theequation{B.\arabic{equation}}
In this part we will investigate POAFD in $H^2(T_\Gamma)$, where $\Gamma$ is a regular cone. Based on the discussions in Section $4$, we know that Lemmas \ref{first-app1} - \ref{third-app1} play important roles in studying POAFD in the octant case. The techniques used in Theorem \ref{thm1} and Lemma \ref{re-ex-lem} still work for the proposed approximation in $H^2(T_{\Gamma})$. Therefore, for a regular cone, we only need to consider the analogous results of Lemmas \ref{first-app1} - \ref{third-app1}. As mentioned in Section $4$, BVC is sufficient for us to obtain POAFD. In this part, under the assumption (\ref{CS-infty}), we will prove certain properties of boundary behavior of functions in $H^p(T_{\Gamma}),1<p<\infty$, which can be regarded as special cases of the analogous results of Lemmas \ref{first-app1} - \ref{third-app1}. When $p=2$, such properties give BVC in $H^2(T_{\Gamma_1})$.

We need the following results as preparation.
\begin{lem}
	Suppose that $\Gamma$ is a regular cone in $\mathbb R^n$. For $z=x+iy\in T_{\Gamma}$ and $1<p\leq \infty$, we have
	\begin{align}\label{P-S-estimate}
	\|P_y(x-\cdot)\|_{L^p(\mathbb R^n)}\leq \frac{1}{2^{\frac{n}{p}}}K(z,\overline z)^{1-\frac{1}{p}}.
	\end{align}
\end{lem}
\begin{proof}
	For $2<p<\infty$ and $z=x+iy\in T_{\Gamma}$, we have
	\begin{align}
	\begin{split}
	\int_{\mathbb R^n}|K(\xi,\overline z)|^p d\xi &=\sup_{\xi\in \mathbb R^n}|K(\xi,\overline z)|^{p-2}\int_{\mathbb R^n}|K(\xi,\overline z)|^2d\xi \\
	&\leq \frac{1}{2^n}K(z,\overline z)^{p-2}K(z,\overline z)=\frac{1}{2^n}K(z,\overline z)^{p-1}.
	\end{split}
	\end{align}
	Hence, for $1<p<\infty,$
	\begin{align}
	\int_{\mathbb R^n}|P_y(x-\xi)|^pd\xi\leq \int_{\mathbb R^n}\frac{|K(\xi,\overline z)|^{2p}}{K(z,\overline z)^p}d\xi\leq \frac{1}{2^n}K(z,\overline z)^{2p-1-p}=\frac{1}{2^n}K(z,\overline z)^{p-1}.
	\end{align}
	Obviously, when $p=\infty$, we have $\|P_y(x-\cdot)\|_{L^\infty(\mathbb R^n)}\leq K(z,\overline z)$.
\end{proof}

Unlike the octant case, estimation of $K(z,\overline z)$ at points of $\partial \Gamma$ and at infinity is not easily accessible. The next lemma gives useful estimates of $K(z,\overline z)$ on $\partial \Gamma$.
\begin{lem}[{\cite[Lemma 2]{AK},\cite[Proposition I.3.2]{JFAK}}]\label{AK_lemma}
	Let $\Gamma$ be a regular cone in $\mathbb R^n, n\geq 3$, and $\beta \in \partial \Gamma$. Then
	\begin{align}\label{AK_formula}
	\lim_{y\in \Gamma, y\to \beta}\int_{\Gamma^{*}}e^{-4y\cdot t} dt=
	\infty.
	\end{align}
\end{lem}
\begin{remark} It is obvious that $(\ref{AK_formula})$ is true when $n=1$. For the self containing purpose we illustrate the proof. For $n=2,$ as shown in \cite{DeCarli} it suffices to verify (\ref{AK_formula}) holds for a special class of regular cones in $\mathbb R^2:$
	\begin{align}\label{2-cone}
	\Gamma^\kappa=\{y\in \mathbb R^2; |y_1|<\kappa y_2\}, \quad 0<\kappa<\infty.
	\end{align}
	The dual cone of $\Gamma^\kappa$ is
	$$
	\Gamma^{\kappa,*}=\{t\in \mathbb R^2; |t_1|<\frac{1}{\kappa}t_2\}.
	$$
	By a direct computation, we have $$K_{\Gamma^\kappa}(z,\overline z)=\int_{\Gamma^{\kappa,*}}e^{-4\pi y\cdot t}dt=\frac{\kappa}{8\pi^2(\kappa^2y_2^2-y_1^2)}$$ (see e.g. \cite{DeCarli} for general $n$). Obviously, for $\beta\in \partial \Gamma^\kappa$, $\lim_{y\in \Gamma^\kappa, y\to \beta}K_{\Gamma^\kappa}(z,\overline z)=\infty$. In the following discussion, we write elements in $\Gamma$ as column vectors.
	If $\Gamma$ is a regular cone in $\mathbb R^2$, then, for $y\in
	\Gamma$, there exists a matrix $P\in SO(2,\mathbb R)=\{Q\in
	GL(2,\mathbb R); QQ^{T}=Q^{T}Q=I, |Q|=1\}$ and $\Gamma^{\kappa}$
	such that
	$$
	(y_1, y_2)^{T}= P(\xi_1, \xi_2)^{T}, \quad
	\xi=(\xi_1, \xi_2)^T\in \Gamma^{\kappa}.
	$$
	Then we
	have
	$$
	K(z,\overline z)= \int_{\Gamma^{*}}e^{-4\pi y\cdot t}dt
	= |P|\int_{\Gamma^{\kappa,*}}e^{-4\pi (P\xi)\cdot (Pt^{\prime})}dt^{\prime}
	= \int_{\Gamma^{\kappa,*}}e^{-4\pi \xi\cdot
		t^{\prime}}dt^{\prime}.
	$$
	Since $P$ is nonsingular, $\xi\to \partial \Gamma^\kappa$ as $y\to \partial \Gamma$. Hence,
	$
	\lim_{y\in\Gamma, y\to \beta} K(z,\overline z)=\infty.
	$
\end{remark}
\medskip

What we are concerned is whether there holds
\begin{align}\label{CS-infty}
\lim_{\widetilde y\in \widetilde{\Gamma},|\widetilde y|\to\infty}K(x+i\widetilde y,\overline{x+i\widetilde y})=0,
\end{align}
where $\widetilde \Gamma=\overline{y^\prime+\Gamma}, y^\prime\in \Gamma.$ In this paper we do not prove that (\ref{CS-infty}) holds for general cones. Nevertheless, we can show that (\ref{CS-infty}) holds for
two classes of cones: the polygonal cones and the circular cones.  If a polygonal cone is the interior of the convex hull of a finite number of $n$ linearly independent rays meeting at the origin, we call it $n$-sided polygonal cone. A polygonal cone, however, is always a finite union of $n$-sided polygonal cones. Hence the general polygonal cone case is reduced to the $n$-sided polygonal cone case. For the circular cone case, we only need to consider the circular cone of the form $\Gamma^\kappa=\{y\in\mathbb R^n;\sqrt{\sum_{j=1}^{n-1}|y_j|^2}<\kappa y_n\},\kappa>0,$ as we can reduce a general circular cone to this form by using some rotation. In fact, these two classes of cones are two different generalizations of $\Gamma^\kappa$ in $\mathbb R^n,n\geq 3.$ For simplicity, we consider $\Gamma^\kappa$ in $\mathbb R^2.$ We will use two different ways to show that (\ref{CS-infty}) holds for $\Gamma^\kappa$. One can then easily conclude that $(\ref{CS-infty})$ holds for the polygonal and the circular cones, in general.

We first show the way that can be utilized in proving that (\ref{CS-infty}) holds for the circular cones. As shown previously, for $y\in \overline {\Gamma^\kappa}$,
\begin{align*}
K_{\Gamma^\kappa}(x+i(y+y^\prime),\overline {x+i(y+y^\prime)}) &=\frac{\kappa}{8\pi^2(\kappa^2(y_2+y_2^\prime)^2-(y_1+y_1^\prime)^2)}\\
&\leq \frac{\kappa}{8\pi^2\delta_0[\kappa(y_2+y_2^\prime)+|y_1+y_1^\prime|]},
\end{align*}
where $\delta_0=dist(\widetilde{\Gamma^\kappa},\Gamma^{\kappa,c})>0$, and $\Gamma^{\kappa,c}$ is the complement of $\Gamma^\kappa$. This implies (\ref{CS-infty}).

Next we show the way that can be utilized in proving that (\ref{CS-infty}) holds for the polygonal cones. Note that there exist a linear transformation $Q$ that maps the first octant onto ${\Gamma^{\kappa}},$ i.e., $\widetilde y=Q\widetilde \xi, \widetilde y\in \Gamma^\kappa,\widetilde \xi\in \Gamma_1$ (see e.g. \cite{SW,Rudin}).
Hence we have
$$
K_{\Gamma^\kappa}(x+i\widetilde y, \overline{x+i\widetilde y})=\int_{\Gamma^{\kappa,*}} e^{-4\pi\widetilde y\cdot t} dt = \frac{1}{|Q|}\int_{\overline\Gamma_1}e^{-4\pi\widetilde\xi\cdot t^\prime} dt^\prime= \frac{1}{|Q|}K_{\Gamma_1}(i\widetilde\xi,\overline{
	i\widetilde \xi}).
$$
Since $|\widetilde\xi|\to\infty$ as $|\widetilde y|\to \infty$, we conclude the desired result again. Since the interior of the dual cone of a polygonal cone $\Gamma$ is polygonal, there exist $n$-sided polygonal cones $\Gamma_{(k)},k=1,...,N,$ such that
$$K_{\Gamma}(w,\overline z)=\sum_{k=1}^N \int_{\Gamma_{(k)}^*}e^{-4\pi(w-\overline z)\cdot t}dt=\sum_{k=1}^N K_{\Gamma_{(k)}}(w,\overline z),$$
where $\Gamma^*=\cup_{k=1}^N \Gamma_{(k)}^*$, and for each $\Gamma_{(k)}$ there exists a linear transformation mapping the first octant onto $\Gamma_{(k)}$. Therefore, we can easily get that (\ref{CS-infty}) holds for the polygonal cones.

For general regular cones, we have the following result, which is closely related to (\ref{CS-infty}), but we note that the lemma is not sufficient to prove (\ref{CS-infty}).
\begin{lem}\label{infty-lem}
	Suppose that $\Gamma_0$ is a regular cone whose closure is contained in $\Gamma\cup \{0\}$, where $\Gamma$ is a regular cone. Then
	$$
	\lim_{y\in \overline \Gamma_0,|y|\to \infty}K(iy,\overline {iy})=0.
	$$
\end{lem}
\begin{proof}
	We first show that
	\begin{align*}
	\lim_{y\in \overline\Gamma_0,|y|\to\infty}e^{-4\pi y\cdot t}=0.
	\end{align*}
	
	We claim that if $\eta\in \overline \Gamma_0$ and $t\in \Gamma^{*}$, then there
	exists a $\delta>0$ such that $\delta|\eta||t|\leq \eta\cdot t$. Denote by $\Sigma$ the set $\{\xi\in\mathbb R^n;|\xi|=1\}$. Define a
	function $H(\eta, t)=\eta\cdot t$, $\eta\in \overline \Gamma_0\cap \Sigma, t\in
	\Gamma^{*}\cap \Sigma$. From the definition of $\Gamma^{*}$ and $\overline\Gamma_0$
	we have $0< \eta\cdot t$. Since $\overline\Gamma_0\cap \Sigma$ and
	$\Gamma^{*} \cap \Sigma$ are both compact, the existence of $\delta
	>0$ follows from the fact that $0<\eta\cdot t=H(\eta,t)$ and $H(\eta,t)$ is a
	continuous function. Consequently, we have
	$$
	\lim_{y\in \overline\Gamma_0, |y|\to \infty}e^{-4\pi y\cdot t}\leq \lim_{y\in \overline\Gamma_0,
		|y|\to \infty}e^{-4\pi \delta |y||t|}=0, \quad t\in \Gamma^{*}
	$$
	and $$e^{-4\pi\delta |y||t|}\leq e^{-4\pi\delta |t|},\quad |y|\geq 1,$$ where $\int_{\Gamma^*}e^{-4\pi\delta|t|}dt<\infty.$
	Therefore, by the Lebesgue dominated convergence theorem we have
	$$
	\lim_{y\in \overline\Gamma_0,|y|\to\infty}K(iy,\overline{iy})=0.
	$$
\end{proof}

\begin{remark}
	On one hand, the argument used in Lemma \ref{infty-lem} can not be applied to $\widetilde\Gamma=\overline{y^\prime+\Gamma}$ with some fixed $y^\prime\in \Gamma$ since the key point of such argument is that the union of all dilations of $\overline \Gamma_0\cap\Sigma$ is $\overline\Gamma_0$ while this is not the fact of $\widetilde\Gamma\cap \Sigma.$ On the other hand, Lemma \ref{infty-lem} shows that (\ref{CS-infty}) holds in most situations. The unsolved situation can be almost concluded as that $\widetilde y\in \partial \widetilde \Gamma,|\widetilde y|\to \infty.$ The assumption (\ref{CS-infty}) should hold for more cones other than those discussed in this paper. For instance, one can easily check that (\ref{CS-infty}) holds for the symmetric cone in $\mathbb R^3$ given by $\{y=(y_1,y_2,y_3)\in\mathbb R^3; y_1>0,y_1y_2-y_3^2>0 \}$ (see \cite{BBGNPR}).
\end{remark}

Under the assumption that (\ref{CS-infty}) holds, the main result of this part is stated as follows.
\begin{thm}\label{g-cone}
	Suppose that $\Gamma$ is a regular cone such that (\ref{CS-infty}) holds. For $F\in H^p(T_\Gamma), 1< p<\infty,$ and $z=x+iy\in T_{\Gamma}$, we have the following results.
	\begin{align}\label{case1-general-cone}
	\lim_{y\in \Gamma,y\to\beta}\frac{|F(z)|}{K(z,\overline z)^{\frac{1}{p}}}=0
	\end{align}
	holds uniformly for $x\in \mathbb R^n$, where $\beta\in \partial \Gamma$.
	\begin{align}\label{case2-general-cone}
	\lim_{y\in \Gamma, |y|\to \infty}\frac{|F(z)|}{K(z,\overline z)^{\frac{1}{p}}}=0
	\end{align}
	holds uniformly for $x\in \mathbb R^n$.
	\begin{align}\label{case3-general-cone}
	\lim_{|x|\to \infty}\frac{|F(z)|}{K(z,\overline z)^{\frac{1}{p}}}=0
	\end{align}
	holds uniformly for $y\in \Gamma$.\\
	In particular, for regular cones in $\mathbb R^2$ (that is, $n=2$), and for polygonal cones and circular cones in general $\mathbb R^n$, we can prove the validity of (\ref{CS-infty}) but do not need to specially assume it.
\end{thm}
\begin{proof}
	By Theorem \ref{g-rf-p}, we have
	$$
	F(z)=\int_{\mathbb R^n}F(\xi)P_y(x-\xi)d\xi,
	$$
	where $F(\xi)\in L^p(\mathbb R^n)$.\\
	Therefore, for any $\epsilon>0$, we can find $G\in L^r(\mathbb R^n)\cap L^p(\mathbb R^n), p<r<\infty,$ such that $$\|F-G\|_{L^p(\mathbb R^n)}<\epsilon.$$
	Hence, for $q$ satisfying $\frac{1}{p}+\frac{1}{q}=1$ and $h$ satisfying $\frac{1}{r}+\frac{1}{h}=1$,
	\begin{align}\label{g-cone-fact}
	\begin{split}
	\frac{|F(z)|}{K(z,\overline z)^{\frac{1}{p}}}&\leq \frac{\int_{\mathbb R^n}|F(\xi)-G(\xi)|P_y(x-\xi)d\xi}{K(z,\overline z)^{\frac{1}{p}}}+\frac{\int_{\mathbb R^n}|G(\xi)|P_y(x-\xi)d\xi}{K(z,\overline z)^{\frac{1}{p}}}\\
	&\leq \|F-G\|_{L^p(\mathbb R^n)}\frac{\|P_y(x-\cdot)\|_{L^q(\mathbb R^n)}}{K(z,\overline z)^{\frac{1}{p}}}+\|G\|_{L^r(\mathbb R^n)}\frac{\|P_y(x-\cdot)\|_{L^h(\mathbb R^n)}}{K(z,\overline z)^{\frac{1}{p}}}\\
	&< \epsilon\frac{\frac{1}{2^{\frac{n}{q}}}K(z,\overline z)^{1-\frac{1}{q}}}{K(z,\overline z)^{\frac{1}{p}}}+\|G\|_{L^r(\mathbb R^n)}\frac{\frac{1}{2^{\frac{n}{h}}}K(z,\overline z)^{1-\frac{1}{h}}}{K(z,\overline z)^{\frac{1}{p}}}\\
	&= \frac{\epsilon}{{2^{\frac{n}{q}}}}+\frac{\|G\|_{L^r(\mathbb R^n)}}{2^{\frac{n}{h}}}K(z,\overline z)^{\frac{1}{r}-\frac{1}{p}},
	\end{split}
	\end{align}
	where the first inequality is given by the triangle inequality, and the second inequality follows from the H\" older inequality.\\
	By Lemma \ref{AK_lemma} and the fact that $\frac{1}{r}-\frac{1}{p}<0$, we have $K(z,\overline z)^{\frac{1}{r}-\frac{1}{p}}\to 0$ as $y\to \beta$. Therefore, we complete the proof of (\ref{case1-general-cone}).
	
	By Theorem \ref{g-rf-p}, we know that
	$$
	F(z)=\int_{\mathbb R^n} F(\xi) P_y(x-\xi)d\xi.
	$$
	Then, for any given $\epsilon>0$, we can find $y^\prime\in \Gamma$ such that
	$$
	\int_{\mathbb R^n} |F(\xi+iy^\prime)-F(\xi)|^p d\xi <\epsilon.
	$$
	So
	\begin{align*}
	\frac{|F(z)|}{{K(z,\overline z)}^{\frac{1}{p}}} &= \frac{|\int_{\mathbb R^n} \left(F(\xi)-F(\xi+iy^\prime)\right)P_y(x-\xi)d\xi|+|\int_{\mathbb R^n} F(\xi+iy^\prime)P_y(x-\xi)d\xi|}{{K(z,\overline z)}^{\frac{1}{p}}}\\
	& < \frac{\epsilon}{2^{\frac{n}{q}}}+ \frac{|\int_{\mathbb R^n} F(\xi)P_{y+y^\prime}(x-\xi)d\xi|}{K(z,\overline z)^{\frac{1}{p}}}.
	\end{align*}
	Similar to (\ref{g-cone-fact}), we have
	\begin{align}\label{g-cone-fact2-ineq}
	\begin{split}
	\frac{|\int_{\mathbb R^n} F(\xi)P_{y+y^\prime}(x-\xi)d\xi|}{K(z,\overline z)^{\frac{1}{p}}} &\leq \frac{\epsilon}{2^\frac{n}{q}} \frac{K(z+iy^\prime,\overline{z+iy^\prime})^{1-\frac{1}{q}}}{K(z,\overline z)^{\frac{1}{p}}}\\
	&+\frac{||G||_{L^r(\mathbb R^n)}}{2^\frac{n}{h}}\frac{K(z+iy^\prime,\overline{z+iy^\prime})^{1-\frac{1}{h}}}{K(z,\overline z)^{\frac{1}{p}}},
	\end{split}
	\end{align}
	where $G\in L^r(\mathbb R^n)\cap L^p(\mathbb R^n), 1<r<p$, $\frac{1}{p}+\frac{1}{q}=1$ and $\frac{1}{r}+\frac{1}{h}=1$.\\
	Notice that $K(z+iy^\prime,\overline{z+iy^\prime})<K(z,\overline z)$. Hence, the first term in (\ref{g-cone-fact2-ineq}) is strictly less than $\frac{\epsilon}{2^{\frac{n}{q}}}$. For the second term in (\ref{g-cone-fact2-ineq}), by $K(z+iy^\prime,\overline{z+iy^\prime})<K(z,\overline z),$ $\lim_{y\in\Gamma,|y|\to \infty}K(z+iy^\prime,\overline {z+iy^\prime})=0$ and $\frac{1}{r}-\frac{1}{p}>0,$ we can easily conclude that, for $|y|$ large enough,
	$$
	\frac{K(z+iy^\prime,\overline{z+iy^\prime})^{1-\frac{1}{h}}}{K(z,\overline z)^{\frac{1}{p}}}<\epsilon.
	$$
	
	By the above discussions, the proof of (\ref{case2-general-cone}) is completed.\\
	To prove (\ref{case3-general-cone}), because of (\ref{case1-general-cone}) and (\ref{case2-general-cone}), we only need to show that
	\begin{align}\label{g-cone-fact21}
	\lim_{|x|\to\infty}\frac{|F(z)|}{K(z,\overline z)^{\frac{1}{p}}}=0
	\end{align}
	holds uniformly for $y\in A_0$, where $A_0$ is a compact subset in $\Gamma$.
	Notice that
	$$
	K(z, \overline z)= \int_{\Gamma^*}e^{-4\pi y\cdot
		t}dt=K(iy, -iy),
	$$
	and $y\in A_0$. It suffices to show
	\begin{align}\label{suff-cond-third-case}
	\lim_{|x|\to \infty} {|F(z)|}=0.
	\end{align}
	Since $A_0$ is compact, there exists a constant $\rho>0$
	such that $d(A_0,{\Gamma}^{c})=\inf\{|y-\xi|;y\in A_{0}, \xi
	\not\in \Gamma\}\geq {\rho}$, where ${\Gamma}^c$ is the complement
	of $\Gamma$. Let $A_1=\overline {\cup_{y\in A_0}\{\eta;
		|\eta-y|<\frac{\rho}{2}\}}$. Obviously, $d(A_1,{\Gamma}^c)\geq
	\frac{\rho}{2}$ and $A_1$ is also compact. Based on the fact that $ \int_{A_1}\int_{\mathbb
		R^n}|F(x+iy)|^p dx dy<\infty$, and the definition of
	functions in $H^p(T_{\Gamma}),$ we have
	\begin{align}\label{case3-fact1}
	\int_{A_1}\int_{\mathbb |x|>N}|F(x+iy)|^p dx dy\to 0, \quad N\to
	\infty.
	\end{align}
	Recall that $|F|^p$ is subharmonic.
	For $z\in T_{\Gamma}$, we have
	\begin{align}\label{case3-fact2}
	|F(x+iy)|^p\leq
	\frac{1}{V(B_z(\frac{\rho}{4}))}\int_{B_z(\frac{\rho}{4})}|F(\xi+i\eta)|^p d\xi
	d\eta,
	\end{align}
	where $V(B_z(\frac{\rho}{4}))$ is the volume of the ball
	$B_z(\frac{\rho}{4})$ centered at $z$ with radius $\frac{\rho}{4}$. From (\ref{case3-fact2}), for $y\in A_0$, we
	have
	\begin{align}\label{case3-fact3}
	\begin{split} |F(x+iy)|^p &\leq
	\frac{1}{V(B_z(\frac{\rho}{4}))}\int_{B_z(\frac{\rho}{4})}|F(\xi+i\eta)|^pd\xi
	d\eta\\
	&\leq \frac{1}{V(B_z(\frac{\rho}{4}))}\int_{\{\eta; |\eta-y|\leq
		\frac{\rho}{4}\}}\int_{\{\xi;|\xi-x|\leq
		\frac{\rho}{4}\}}|F(\xi+i\eta)|^pd\xi
	d\eta\\
	&\leq
	\frac{1}{V(B_z(\frac{\rho}{4}))}\int_{A_1}\int_{\{\xi;|\xi-x|\leq
		\frac{\rho}{4}\}}|F(\xi+i\eta)|^pd\xi
	d\eta.\\
	\end{split}
	\end{align}
	Since $|\xi-x|\leq \frac{\rho}{4}$, we have $|x|-\frac{\rho}{4}\leq
	|\xi| \leq |x|+\frac{\rho}{4}$. Therefore,
	when $|x|>N+\frac{\rho}{4}$, by (\ref{case3-fact1}) we have
	(\ref{case3-fact3}) tends to $0$ uniformly for $y\in A_0$. The proof is completed.
\end{proof}

\begin{remark}
	(1) When $p=2$, Theorem \ref{g-cone} implies existence of $z^{(m+1)}$ in the following minimization problem
	\begin{align}\label{min_pro_revise1}
	z^{(m+1)} :=\arg \max_{z\in \Gamma}|\langle F,\mathcal B_{m+1}^z\rangle|.
	\end{align}
	Hence, we obtain POAFD in $H^2(T_\Gamma)$ if $\Gamma$ is one of the following cases: a regular cone in $\mathbb R^2$; a polygonal cone in $\mathbb R^n$; a circular cone in $\mathbb R^n.$\\
	(2) Since $|F|^p$ is still subharmonic for $0<p\leq 1,$ the technique used in proving (\ref{suff-cond-third-case}) in Theorem \ref{g-cone} still works. Thus (\ref{suff-cond-third-case}) holds for $0<p<\infty$. Note that Theorem \ref{g-cone} plays an essential role in studying POAFD. Moreover, Theorem \ref{g-cone} indeed is analogous to the known result in the Hardy spaces on the unit ball $H^p(\mathbb B_n), 1<p<\infty$ (cf. \cite[page 123]{Z1}),
	$$
	\lim_{|z|\to 1^{-}}(1-|z|^2)^{\frac{n}{p}}|F(z)|=0, \quad F\in H^p(\mathbb B_n),
	$$
	where $K_{\mathbb B_n}(w,\overline z)=\frac{1}{(1-\sum_{k=1}^n w_k\overline {z_k})^n}$ is the Cauchy-Szeg\"o kernel for $H^2(\mathbb B_n).$
\end{remark}

\end{appendix}

\end{document}